\documentclass[english,1p,authoryear]{elsarticle}
\usepackage[T1]{fontenc}
\usepackage[latin9]{inputenc}
\usepackage{float}
\usepackage{amsthm}
\usepackage{amsmath}
\usepackage{amssymb}
\providecommand{\tabularnewline}{\\}
\floatstyle{ruled}
\newfloat{algorithm}{tbp}{loa}
\floatname{algorithm}{Algorithm}

\theoremstyle{plain}
\theoremstyle{plain}
\newtheorem{thm}{Theorem}
\newenvironment{lyxlist}[1]
{\begin{list}{}
{\settowidth{\labelwidth}{#1}
 \setlength{\leftmargin}{\labelwidth}
 \addtolength{\leftmargin}{\labelsep}
 \renewcommand{\makelabel}[1]{##1\hfil}}}
{\end{list}}
  \theoremstyle{definition}
  \newtheorem{defn}[thm]{Definition}
  \theoremstyle{remark}
  \newtheorem*{rem*}{Remark}
 \theoremstyle{definition}
  \newtheorem{example}[thm]{Example}
  \theoremstyle{plain}
  \newtheorem{prop}[thm]{Proposition}
  \theoremstyle{plain}
  \newtheorem{lem}[thm]{Lemma}
  \theoremstyle{plain}
  \newtheorem{cor}[thm]{Corollary}

\usepackage[pdftex]{color}
\usepackage[noend]{morealgorithmic}
\renewcommand\algorithmiccomment[1]{\textcolor{commentcolor}{ --- #1}}
\renewcommand\algorithmicprint{\textbf{warn}}
\renewcommand\algorithmicdo{}
\renewcommand\algorithmicthen{}

\newcommand\SPol{\mathrm{Spol}}
\newcommand\lcm{\mathrm{lcm}}
\newcommand\lt{\mathrm{HM}}
\newcommand\lc{\mathrm{HC}}
\newcommand\ltmul[1]{\sigma_{#1}}

\newcommand\ideal[1]{\ensuremath{\left<#1\right>}}

\definecolor{commentcolor}{rgb}{0.6, 0.6, 0.6}
\definecolor{darklinkcolor}{rgb}{0,0,0.4}
\newcommand\comment[1]{\textcolor{commentcolor}{ --- #1}}

\newcommand\termset{\ensuremath{\mathbb{M}}}
\newcommand\ordlt{\ensuremath{<_T}}
\newcommand\ordle{\ensuremath{\leq_T}}

\newcommand\ordge{\ensuremath{\geq_T}}
\newcommand\siglt{\ensuremath{\prec}}
\newcommand\siggt{\ensuremath{\succ}}

\newcommand\polyring{\ensuremath{\mathcal{R}}}

\newcommand\algindent{\hspace{2em}}

\newcommand\lp{\ensuremath{\textit{r}}}
\newcommand\rules{\ensuremath{\textit{Rule}}}
\newcommand\ctr{\ensuremath{\textit{ctr}}}
\newcommand\bases{\ensuremath{\textit{G}}}
\newcommand\prevbasis{\ensuremath{\bases_{\mathrm{prev}}}}
\newcommand\currbasis{\ensuremath{\bases_{\mathrm{curr}}}}

\newcommand\curridx{\ensuremath{\mathit{curr\_idx}}}
\newcommand\critpairs{\ensuremath{P}}
\newcommand\spols{\ensuremath{S}}
\newcommand\reducedpols{\ensuremath{R}}
\newcommand\newpols{\ensuremath{S}}

\newcommand\completed{\ensuremath{\textit{completed}}}
\newcommand\todo{\ensuremath{\textit{to\_do}}}
\newcommand\newlycompleted{\ensuremath{\textit{newly\_completed}}}
\newcommand\redo{\ensuremath{\textit{redo}}}

\newcommand\reducedbasis{\ensuremath{B}}

\newcommand\poly{\ensuremath{\mathrm{Poly}}}

\newcommand\newlist[1]{\ensuremath{\mathrm{List}\left({#1}\right)}}

\newcommand\basisoriginal{\ensuremath{\mathrm{\textsc{Basis}}}}
\newcommand\basis{\ensuremath{\mathrm{\textsc{Basis/C}}}}
\newcommand\partialbasis{\ensuremath{\mathrm{\textsc{Incremental\_Basis}}}}
\newcommand\partialbasisc{\ensuremath{\mathrm{\textsc{Incremental\_Basis/C}}}}
\newcommand\critpair{\ensuremath{{\mathrm{\textsc{Critical\_Pair}}}}}
\newcommand\spol{\ensuremath{{\mathrm{\textsc{Compute\_SPols}}}}}
\newcommand\reduction{\ensuremath{{\mathrm{\textsc{Reduction}}}}}
\newcommand\sig{\ensuremath{\mathrm{Sig}}}
\newcommand\indexvar{\ensuremath{\nu}}
\newcommand\otherindexvar{\ensuremath{\nu'}}
\newcommand\multiplier{\ensuremath{\tau}}
\newcommand\othermultiplier{\ensuremath{\tau'}}
\newcommand\basisvar{\ensuremath{\mathbf{F}}}
\newcommand\sigvar[1]{\sigformat{\multiplier_{#1}}{\indexvar_{#1}}}
\newcommand\zerosig{\ensuremath{\mathbf{0}}}
\newcommand\sigformat[2]{\ensuremath{{#1}\basisvar_{#2}}}
\newcommand\rewritable{\ensuremath{{\mathrm{\textsc{Is\_Rewritable}}}}}
\newcommand\findrewriting{\ensuremath{{\mathrm{\textsc{Find\_Rewriting}}}}}
\newcommand\addrule{\ensuremath{{\mathrm{\textsc{Add\_Rule}}}}}
\newcommand\normalform{\ensuremath{\mathrm{Normal\_Form}}}
\newcommand\topreduction{\ensuremath{{\mathrm{\textsc{Top\_Reduction}}}}}
\newcommand\createreducedbasis{\ensuremath{\mathrm{\textsc{Setup\_Reduced\_Basis}}}}
\newcommand\findreductor{\ensuremath{{\mathrm{\textsc{Find\_Reductor}}}}}
\newcommand\closer{\ \scalebox{0.4}{\begin{picture}(0,0)
    \thicklines
    \put(0,10){\line(1,-1){10}}
    \put(0,10){\line(1,1){10}}
    \put(10,0){\line(1,1){10}}
    \put(10,20){\line(1,-1){10}}
    \put(15,0){\line(1,1){10}}
    \put(15,20){\line(1,-1){10}}
  \end{picture}}}

\newcommand\altreduction{\ensuremath{{\mathrm{\textsc{ Reduction}}}}}

\newcommand\safereducers{\ensuremath{\textit{safe}}}

\usepackage{babel}

\begin{document}

\title{F5C: a variant of Faug\`ere's F5 algorithm with reduced Gr\"obner
bases}

\cortext[corauth]{Corresponding author}

\author[ceder]{Christian Eder\corref{corauth}}

\address[ceder]{Universit\"at Kaiserslautern}

\ead{ederc@mathematik.uni-kl.de}

\author[jperry]{John Perry}

\address[jperry]{The University of Southern Mississippi}

\ead{john.perry@usm.edu}

\ead[jperry]{www.math.usm.edu/perry}
\begin{abstract}
Faug\`ere's F5 algorithm computes a Gr\"obner basis incrementally,
by computing a sequence of (non-reduced) Gr\"obner bases. The authors
describe a variant of F5, called F5C, that replaces each intermediate
Gr\"obner basis with its reduced Gr\"obner basis. As a result, F5C
considers fewer polynomials and performs substantially fewer polynomial
reductions, so that it terminates more quickly. We also provide a
generalization of Faug\`ere's characterization theorem for Gr\"obner
bases.
\end{abstract}
\maketitle

\section{Introduction}

Gr\"obner bases, first introduced in \citep{Buchberger65}, are by
now a fundamental tool of computational algebra, and Faug\`ere's
F5 algorithm is noted for its success at computing certain difficult
Gr\"obner bases \citep{Fau02Corrected,FrobeniusInF5,Fau05CStar}.
The algorithm's design is incremental: given a list of polynomials
$F=\left(f_{1},\ldots,f_{m}\right)$, F5 computes for each $i=2,\ldots,m$
a Gr\"obner basis $G_{i}$ of the ideal $\ideal{F_{i}}=\ideal{f_{1},\ldots,f_{i}}$
using a Gr\"obner basis $G_{i-1}$ of the ideal $\ideal{F_{i-1}}$.
The algorithm assigns each polynomial $p$ a ``signature'' determined
by how it computed $p$ from $F$; using the signature, F5 detects
a large number of zero reductions, and sometimes avoids these costly
computations altogether.

This paper considers the challenge of modifying F5 so that it replaces
$G_{i-1}$ with its \emph{reduced} Gr\"obner basis $\reducedbasis_{i-1}$
before proceeding to $\ideal{F_{i}}$. Working with the reduced Gr\"obner
basis is desirable because each stage of the pseudocode of \citep{Fau02Corrected}
usually generates many polynomials that are not needed for the Gr\"obner
basis property, and there is no interreduction between stages. In
one example, we show that a straightforward implementation of the
pseudocode of \citep{Fau02Corrected} on Katsura-9 concludes with
a Gr\"obner basis where nearly a third of the polynomials are unnecessary.

Stegers introduces a variant that uses $\reducedbasis_{i-1}$ to \emph{reduce}
newly computed generators of $\left\langle F_{i}\right\rangle $ \citep{Stegers2006}.
We call this variant F5R, for ``F5 Reducing by reduced Gr\"obner
bases.'' However, F5R still uses the unreduced basis $G_{i-1}$ to
compute critical pairs and new polynomials for $G_{i}$. As Stegers
points out, discarding $G_{i-1}$ in favor of $\reducedbasis_{i-1}$
is not a casual task, since the signatures of $G_{i-1}$ do not correspond
to the polynomials of $\reducedbasis_{i-1}$.

The solution we propose is to generate new signatures that correspond
to $B_{i-1}$, which generates the same ideal as $F_{i-1}$. With
this change, we can discard $G_{i-1}$ completely. The modified algorithm
generates fewer polynomials and performs fewer reduction operations.
Naturally, this means that the new variant consumes less CPU time,
as documented in two different implementations. Although it is a non-trivial
variant of F5, it respects its ancestor's elegant structure and modifies
only one subalgorithm. We call this variant F5C, for ``F5 Computing
by reduced Gr\"obner bases.''

After a review of preliminaries in Section~\ref{sec:Background},
we describe F5C in Section~\ref{sec:Modification}, and provide some
run-time data. A preliminary implementation in \textsc{Singular} is
complete~\citep{Singular3}, and we present comparative timings for
F5, F5R, and F5C. A proof of correctness appears in Section~\ref{sec: F5 correct},
and in Section~\ref{sub:Principal-syzygies} we show that one of
Faug\`ere's criteria is a a special case of a more general criterion.

The authors have made available a prototype implementation of F5,
F5R, and F5C as a \textsc{Singular} library \citep{Singular3,GreuelPfister2008}
at

\begin{center}
\texttt{http://www.math.usm.edu/perry/Research/f5\_library.lib}\textsc{
.}
\par\end{center}

\noindent A prototype implementation for the Sage computer algebra
system \citep{sage} developed by Martin Albrecht, with some assistance
from the authors, is available at

\begin{center}
\texttt{http://bitbucket.org/malb/algebraic\_attacks/src/tip/f5.py} .
\par\end{center}

\noindent This latter implementation can use F4-style reduction.

\section{\label{sec:Background}Background Material}

This section describes the fundamental notions and the conventions
in this paper. Our conventions differ somewhat from Faug\`ere's,
partly because the ones here make it relatively easy to describe and
implement the variant F5C.

Let $\mathbb{F}$ be a field and $\polyring=\mathbb{F}\left[x_{1},x_{2},\ldots,x_{n}\right]$.
Let $\ordlt$ denote a fixed admissible ordering on the monomials
$\mathbb{M}$ of $\polyring$. For every polynomial $p\in\polyring$
we denote the head monomial of $p$ with respect to $\ordlt$ by $\lt\left(p\right)$
and the head coefficient with respect to $\ordlt$ by $\lc\left(p\right)$.
(For us, a monomial has no coefficient.) Let $F=\left(f_{1},f_{2},\ldots,f_{m}\right)\in\polyring^{m}$.
The goal of F5 is to compute a \emph{Gr\"obner basis} of the ideal
$I=\ideal{F}$ with respect to $\ordlt$.

\subsection{\label{sub: Groebner bases}Gr\"obner bases}

A \emph{Gr\"obner basis} of $I$ with respect to $\ordlt$ is a finite
list $G$ of polynomials in $I$ that satisfies the properties $\ideal{G}=I$
and for every $p\in I$ there exists $g\in G$ satisfying $\lt\left(g\right)\mid\lt\left(p\right)$.
If, in addition, every $g\in G$ is monic and has no monomial that is divisible
by $\lt(h)$ for any $h\in G$, then $G$ is a \emph{reduced Gr\"obner basis}.
A reduced Gr\"obner basis exists for any ideal of $\polyring$, and Buchberger
first found an algorithm to compute such a basis~\citep{Buchberger65}.
We can describe Buchberger's algorithm in the following way: set $G=F$,
then iterate the following three steps.
\begin{itemize}
\item Choose a \emph{critical pair} $p,q\in G$ that has not yet been considered,
and construct its \emph{$S$-poly\-nomial}\[
S=\SPol\left(p,q\right)=\lc\left(q\right)\ltmul{p,q}\cdot p-\lc\left(p\right)\ltmul{q,p}\cdot q\]
where\[
\ltmul{p,q}=\frac{\lcm\left(\lt\left(p\right),\lt\left(q\right)\right)}{\lt\left(p\right)}\quad\mbox{and}\quad\ltmul{q,p}=\frac{\lcm\left(\lt\left(p\right),\lt\left(q\right)\right)}{\lt\left(q\right)}.\]
We call $p$ and $q$ the \emph{generators} of $S$ and $\ltmul{p,q}\cdot p$
and $\ltmul{q,p}\cdot q$ the \emph{components} of $S$.
\item \emph{Top-reduce} $S$ with respect to $G$. That is, $r_0=S$,
and while $t=\lt\left(r_i\right)$
remains divisible by $u=\lt\left(g\right)$ for some $g\in G$, put
$r_{i+1}:=r_i-\frac{\lc\left(r_i\right)}{\lc\left(g\right)}\frac{t}{u}\cdot g$.
\item If top-reduction of $S$ terminates after $j$ iterations,
no more top-reductions of $r_j$ are possible, so either $r_j=0$ or
$\lt\left(r_j\right)$ is no longer divisible by $\lt\left(g\right)$
for any $g\in G$.

\begin{itemize}
\item In the first case, we say that $\SPol\left(p,q\right)$\emph{ reduces
to zero with respect to} $G$.
\item In the second case, we say that $S$ top-reduces to $r_j$,
and append $r_j$ to $G$.
The new entry in $G$ means that $\SPol\left(p,q\right)$ now reduces
to zero with respect to $G$.
\end{itemize}
\end{itemize}
The algorithm terminates once the \emph{$S$}-polynomials of all pairs
$p,q\in G$ top-reduce to zero. That this occurs despite the introduction
of new critical pairs when $S$ does not reduce to zero is a well-known
consequence of the Ascending Chain Condition \citep{BWK93,CLO97}.

In this paper we consider several kinds of \emph{representations}
of a polynomial. Let $G$ and $\mathbf{h}$ be lists of $m$ elements
of $\polyring$, $p\in\ideal{G}$, and $t\in\mathbb{M}$. We say that
\begin{itemize}
\item $\mathbf{h}$ is a \emph{$G$-representation} of $p$ if $p=h_{1}g_{1}+\cdots+h_{m}g_{m}$;
\item $\mathbf{h}$ is a \emph{$t$-representation of $p$ with respect
to $G$} if $\mathbf{h}$ is a $G$-representation and for all $k=1,\ldots,m$
we have $h_{k}=0$ or $\lt\left(h_{k}g_{k}\right)\ordle t$; and
\item $\mathbf{h}$ is an \emph{$S$-re\-pre\-sen\-ta\-tion} of $S=\SPol\left(g_{i},g_{j}\right)$
\emph{with respect to $G$} if $\mathbf{h}$ is a $t$-re\-pre\-sen\-ta\-tion
of $S$ with respect to $G$ for some monomial $t\ordlt\lcm\left(\lt\left(g_{i}\right),\lt\left(g_{j}\right)\right)$.
\end{itemize}
We generally omit the phrase ``with respect to $G$'' when it is clear
from context.

If $p$ top-reduces to zero with respect to $G$, then it is easy
to derive an $\lt\left(p\right)$-repre\-sen\-tation of $p$, although
the converse is not always true. Correspondingly, if $p$ is an $S$-poly\-nomial
and $p$ top-\-reduces to zero with respect to $G$, then there exists
an $S$-re\-pre\-sen\-ta\-tion of $p$.

Theorem \ref{thm:equiv cond for GB} summarizes three important characterizations
of a Gr\"obner basis; (C) is from~\citet{Buchberger65}, while (D)
is from~\citet{Lazard83}. The proof, and many more characterizations
of a Gr\"obner basis, can be found in~\citep{BWK93}.
\begin{thm}
\label{thm:equiv cond for GB}Let $G$ be a finite list of polynomials
in $\polyring$, and $\ordlt$ an ordering on the monomials of $\polyring$.
The following are equivalent:
\begin{lyxlist}{(D)}
\item [{(A)}] $G$ is a Gr\"obner basis with respect to $\ordlt$.
\item [{(B)}] For all nonzero $p\in\ideal{G}$ there exists $g\in G$ such
that $\lt\left(g\right)\mid\lt\left(p\right)$.
\item [{(C)}] For all $p,q\in G$ $\SPol\left(p,q\right)$ top-reduces
to zero with respect to $G$.
\item [{(D)}] For all $p,q\in G$ $\SPol\left(p,q\right)$ has an $S$-representation
with respect to $G$.
\end{lyxlist}
\end{thm}

\subsection{\label{sec: F5}The F5 Algorithm}

In this section we give a brief overview of F5 (Algorithms~\ref{alg:F5}--\ref{alg:Find_Rewriting}).
To make the presentation of F5R and F5C easier, we have made some
minor modifications to the pseudocode of \citet{Fau02Corrected,Stegers2006},
but they are essentially equivalent.

The F5 algorithm \citep{Fau02Corrected} consists of several subalgorithms.
\begin{algorithm}
\caption{\label{alg:F5}$\basisoriginal$}

\begin{algorithmic}[1]

\GLOBALS$\lp$, $\rules$, $\ordlt$

\INPUTS

\STATE$F=\left(f_{1},f_{2},\ldots,f_{m}\right)\in\polyring^{m}$
(homogeneous)

\STATE$<$, an admissible ordering

\ENDINPUTS

\OUTPUTS

\STATE a Gr\"obner basis of $F$ with respect to $<$

\ENDOUTPUTS

\BODY

\STATE$\ordlt:=<$

\STATE Sort $F$ by increasing total degree, breaking ties by increasing
head monomial

\COMMENT{Initialize the record keeping.}

\STATE $\rules:=\newlist{\newlist{}}$

\STATE $\lp:=\newlist{}$

\COMMENT{Compute the basis of $\ideal{f_{1}}$.}

\STATE Append $\left(\sigformat{}{1},f_{1}\cdot\lc\left(f_{1}\right)^{-1}\right)$
to $\lp$

\STATE $\prevbasis=\left\{ 1\right\} $

\STATE $\reducedbasis=\left\{ f_{1}\right\} $

\COMMENT{Compute the bases of $\ideal{f_{1},f_{2}}$, \ldots{},
$\ideal{f_{1},f_{2},\ldots,f_{m}}$.}

\STATE $i:=2$

\WHILE{ $i\leq m$}

\STATE Append $\left(\basisvar_{i},f_{i}\cdot\lc\left(f_{i}\right)^{-1}\right)$
to $\lp$

\STATE $\currbasis:=\partialbasis\left(i,\reducedbasis,\prevbasis\right)$

\IF{$\exists\lambda\in\currbasis$ such that $\poly\left(\lambda\right)=1$}

\RETURN $\left\{ 1\right\} $

\ENDIF

\STATE $\prevbasis:=\currbasis$

\STATE $\reducedbasis:=\left\{ \poly\left(\lambda\right):\;\lambda\in\prevbasis\right\} $\label{line: definition of B}

\STATE $i:=i+1$

\ENDWHILE

\RETURN \textbf{$\reducedbasis$}

\ENDBODY

\end{algorithmic}
\end{algorithm}

\begin{itemize}
\item The entry point is the  $\basisoriginal$. It expects as input a list
of homogeneous polynomials of $\polyring$. $\basisoriginal$ invokes
 $\partialbasis$ to construct Gr\"obner bases of the ideals $\ideal{F_{2}}$,
$\ideal{F_{3}}$, \ldots{}, $\ideal{F_{m}}$, in succession. (Computing
the Gr\"obner basis of $\ideal{F_{1}}$ is trivial.) Polynomials
are stored in a data structure $\lp$, whose details we consider in
Section~\ref{sub:Signatures-and-labeled}. The sets $\currbasis,\prevbasis\subset\mathbb{N}$
index elements of $r$ that correspond to the generators of $\ideal{F_{i}}$
and a Gr\"obner basis of $\ideal{F_{i-1}}$, respectively.%
\begin{algorithm}
\caption{\label{alg:Partial_Basis}$\partialbasis$}

\begin{algorithmic}[1]

\GLOBALS $\lp$, $\ordlt$

\INPUTS

\STATE $i\in\mathbb{N}$

\STATE $\reducedbasis$, a Gr\"obner basis of $\left(f_{1},f_{2},\ldots,f_{i-1}\right)$
with respect to $\ordlt$

\STATE $\prevbasis\subset\mathbb{N}$, indices in $\lp$ of $\reducedbasis$

\ENDINPUTS

\OUTPUTS

\STATE $\currbasis$, indices in $\lp$ of a Gr\"obner basis of
$\left(f_{1},f_{2},\ldots,f_{i}\right)$ with respect to $\ordlt$

\ENDOUTPUTS

\BODY

\STATE $\curridx:=\#\lp$

\STATE $\currbasis:=\prevbasis\cup\left\{ \curridx\right\} $

\STATE Append $\newlist{}$ to $\rules$

\STATE $\critpairs:=\bigcup_{j\in\prevbasis}\critpair\left(\curridx,j,i,\prevbasis\right)$

\WHILE{$\critpairs\neq\emptyset$}

\STATE $d:=\min\left\{ \deg t:\left(t,k,u,\ell,v\right)\in\critpairs\right\} $
\COMMENT{See Algorithm~\ref{alg:Critical_Pair} for structure of $p\in\critpairs$}

\STATE $\critpairs_{d}:=\left\{ \left(t,k,u,\ell,v\right)\in\critpairs:d=\deg t\right\} $

\STATE $\critpairs:=\critpairs\backslash\critpairs_{d}$

\STATE $\spols:=\spol\left(\critpairs_{d}\right)$

\STATE $\reducedpols:=\reduction\left(\spols,\reducedbasis,\prevbasis,\currbasis\right)$

\FOR{$k\in\reducedpols$}

\STATE $\critpairs:=\critpairs\cup\left(\bigcup_{j\in\currbasis}\critpair\left(k,j,i,\prevbasis\right)\right)$\label{line:partialbasis generates new critical pairs}

\STATE $\currbasis:=\currbasis\cup\left\{ k\right\} $

\ENDFOR

\ENDWHILE

\RETURN{$\currbasis$}

\ENDBODY

\end{algorithmic}
\end{algorithm}

\item The goal of  $\partialbasis$ is to compute a Gr\"obner basis of
$\ideal{F_{i}}$ by computing $d$-Gr\"obner bases for $d=1,2,\ldots$.
(A $d$-Gr\"obner basis is one for which all $S$-poly\-nomials
of homogeneous degree at most $d$ reduce to zero; see~\citep{BWK93}.)
$\partialbasis$ iterates the following steps, which follow the general
outline of Buchberger's Algorithm:

\begin{itemize}
\item Generate a list of critical pairs by iterating  $\critpair$ on all
of the pairs of $\left\{ \curridx\right\} \times\prevbasis$. (In
our implementation, $\curridx$ is the location in $\lp$ where $f_{i}$
is stored.)%
\begin{algorithm}
\caption{\label{alg:Critical_Pair}$\critpair$}

\begin{algorithmic}[1]

\GLOBALS $\ordlt$

\INPUTS

\STATE $k,\ell\in\mathbb{N}$ such that $1\leq k<\ell\leq\#\lp$

\STATE $i\in\mathbb{N}$

\STATE $\prevbasis\subset\mathbb{N}$, indices in $\lp$ of a Gr\"obner
basis of $\left(f_{1},f_{2},\ldots,f_{i-1}\right)$ w/respect to $\ordlt$

\ENDINPUTS

\OUTPUTS

\STATE $\left\{ \left(t,u,k,v,\ell\right)\right\} $, corresponding
to a critical pair $\left\{ k,l\right\} $ necessary for

\STATE \algindent the computation of a Gr\"obner basis of $\left(f_{1},f_{2},\ldots,f_{i}\right)$;
$\emptyset$ otherwise

\ENDOUTPUTS

\BODY

\STATE $t_{k}:=\lt\left(\poly\left(k\right)\right)$

\STATE $t_{\ell}:=\lt\left(\poly\left(\ell\right)\right)$

\STATE $t:=\lcm\left(t_{k},t_{\ell}\right)$

\STATE $u_{1}:=t/t_{k}$

\STATE $u_{2}:=t/t_{\ell}$

\STATE $\sigvar{1}:=\sig\left(k\right)$

\STATE $\sigvar{2}:=\sig\left(\ell\right)$

\IF{$\indexvar_{1}=i$ \BAND\  $u_{1}\cdot\multiplier_{1}$ is top-reducible
by $\prevbasis$\label{line: critpair, PS-reducible 1}}

\RETURN{$\emptyset$}

\ENDIF

\IF{$\indexvar_{2}=i$ \BAND\  $u_{2}\cdot\multiplier_{2}$ is top-reducible
by $\prevbasis$\label{line: critpair, PS-reducible 2}}

\RETURN{$\emptyset$}

\ENDIF

\IF{$u_{1}\cdot\sig\left(k\right)\siglt u_{2}\cdot\sig\left(\ell\right)$}

\STATE Swap $u_{1}$ and $u_{2}$

\STATE Swap $k$ and $\ell$

\ENDIF

\RETURN{$\left\{ \left(t,k,u_{1},\ell,u_{2}\right)\right\} $}

\ENDBODY

\end{algorithmic}
\end{algorithm}

\item Identify the critical pairs of smallest degree, and compute the necessary
$S$-poly\-nomials of smallest degree using  $\spol$.%
\begin{algorithm}
\caption{\label{alg:SPols}$\spol$}

\begin{algorithmic}[1]

\GLOBALS $\lp$, $\ordlt$

\INPUTS

\STATE $\critpairs$, a set of critical pairs in the form $\left(t,k,u,\ell,v\right)$

\ENDINPUTS

\OUTPUTS

\STATE $\spols$, a list of indices in $\lp$ of $S$-poly\-nomials
computed

\STATE for a Gr\"obner basis of $\left(f_{1},f_{2},\ldots,f_{i}\right)$

\ENDOUTPUTS

\BODY

\STATE $\newpols:=\left(\right)$

\FOR{$\left(t,k,u,\ell,v\right)\in\critpairs$, from smallest to
largest lcm}

\IF{\NOT\ $\rewritable\left(u,k\right)$ \BAND\ \NOT\  $\rewritable\left(v,\ell\right)$\label{line: spol rewritiable?}}

\STATE Compute $s$, the $S$-polynomial of $\poly\left(k\right)$
and $\poly\left(\ell\right)$

\STATE \label{line: create new poly from S-poly}Append $\left(u\cdot\sig\left(k\right),s\right)$
to $\lp$

\STATE $\addrule\left(u\cdot\sig\left(k\right),\#\lp\right)$

\IF{$s\neq0$}\label{line: if nonzero}

\STATE Append $\#\lp$ to $\newpols$\label{line: if nonzero add}

\ENDIF

\ENDIF

\ENDFOR

\STATE Sort $\newpols$ by increasing signature

\RETURN $\newpols$

\ENDBODY

\end{algorithmic}
\end{algorithm}

\item Top-reduce by passing the output $\spols$ of $\spol$ to  $\reduction$.
\item The output $\reducedpols$ of $\reduction$ indexes those polynomials
that did not reduce to zero; new critical pairs are generated by iterating
$\critpair$ on all pairs $\left(k,j\right)\in\reducedpols\times\currbasis$,
and $\reducedpols$ is appended to $\currbasis$.
\end{itemize}
\end{itemize}
\noindent We higlight the major differences between these subalgorithms
and their counterparts in Buchberger's algorithm:
\begin{itemize}
\item $\critpair$ discards any pair whose corresponding $S$-polynomial
has a component that satisfies the ``new criterion'' of \citep{Fau02Corrected},
described in Section~\ref{sub:Principal-syzygies}.
\item $\spol$ disregards any $S$-polynomial with a ``rewritable'' component,
as described in Section~\ref{sub:Rewritable-polynomials}.
\item $\reduction$ iterates over the most recently computed $S$-poly\-nomials,
from lowest signature to highest. For each $k$ in its input, it:%
\begin{algorithm}
\caption{\label{alg:Reduction}$\reduction$}

\begin{algorithmic}[1]

\GLOBALS $\lp$, $\ordlt$

\INPUTS

\STATE $\newpols$, a list of indices of polynomials added to the
generators $\bases_{i}$

\STATE $\reducedbasis$, a Gr\"obner basis of $\left(f_{1},f_{2},\ldots,f_{i-1}\right)$
with respect to $\ordlt$

\STATE $\prevbasis\subset\mathbb{N}$, indices in $\lp$ corresponding
to $\reducedbasis$

\STATE $\currbasis\subset\mathbb{N}$, indices in $\lp$ of a list
of generators of the ideal of $\left(f_{1},f_{2},\ldots,f_{i}\right)$

\ENDINPUTS

\OUTPUTS

\STATE $\completed$, a subset of $\bases$ corresponding to (mostly)
top-reduced polynomials

\ENDOUTPUTS

\BODY

\STATE $\todo:=\newpols$

\STATE $\completed:=\emptyset$

\WHILE{$\todo\neq()$}

\STATE Let $k$ be the element of $\todo$ such that $\sig\left(k\right)$
is minimal.

\STATE $\todo:=\todo\backslash\left\{ k\right\} $

\STATE $h:=\normalform\left(\poly\left(k\right),\reducedbasis,\ordlt\right)$\label{line: normal form}

\STATE $\lp_{k}:=\left(\sig\left(k\right),h\right)$

\STATE $\newlycompleted,\redo:=\topreduction\left(k,\prevbasis,\currbasis\cup\completed\right)$

\STATE $\completed:=\completed\cup\newlycompleted$

\FOR{$j\in\redo$\label{line: for to add new polys}}

\STATE Insert $j$ in $\todo$ , sorting by increasing signature\label{line: add new polys}

\ENDFOR

\ENDWHILE

\RETURN{$\completed$}

\ENDBODY

\end{algorithmic}
\end{algorithm}
\begin{algorithm}
\caption{\label{alg:Top_Reduction}$\topreduction$}

\begin{algorithmic}[1]

\GLOBALS $\lp$, $\ordlt$

\INPUTS

\STATE $k$, the index of a labeled polynomial

\STATE $\prevbasis\subset\mathbb{N}$, indices in $\lp$ of a Gr\"obner
basis of $\left(f_{1},f_{2},\ldots,f_{i-1}\right)$ w/respect to $\ordlt$

\STATE $\currbasis\subset\mathbb{N}$, indices in $\lp$ of a list
of generators of the ideal of $\left(f_{1},f_{2},\ldots,f_{i}\right)$

\ENDINPUTS

\OUTPUTS

\STATE $\completed$, which has value $\left\{ k\right\} $ if $\lp_{k}$
was \textbf{not} top-reduced and $\emptyset$ otherwise

\STATE $\todo$, which has value

\STATE \algindent $\emptyset$ if $\lp_{k}$ was not top-reduced,

\STATE \algindent $\left\{ k\right\} $ if $\lp_{k}$ is replaced
by its top-reduction, and

\STATE \algindent $\left\{ k,\#\lp\right\} $ if top-reduction of
$\lp_{k}$ generates a polynomial with a signature larger than $\sig\left(k\right)$.

\ENDOUTPUTS

\BODY

\IF{$\poly\left(k\right)=0$}\label{line: if poly(k) =00003D 0}

\PRINT ``Reduction to zero!''

\RETURN $\emptyset,\emptyset$\label{line: zero reduction, return 0,0}

\ENDIF

\STATE $p:=\poly\left(k\right)$

\STATE $J:=\findreductor\left(k,\prevbasis,\currbasis\right)$

\IF{$J=\emptyset$}

\STATE $\lp_{k}:=\left(\sig\left(k\right),p\cdot\left(\lc\left(p\right)\right)^{-1}\right)$

\RETURN $\left\{ k\right\} ,\emptyset$

\ENDIF

\STATE Let $j$ be the single element in $J$

\STATE $q:=\poly\left(j\right)$

\STATE $u:=\frac{\lt\left(p\right)}{\lt\left(q\right)}$

\STATE $c:=\lc\left(p\right)\cdot\left(\lc\left(q\right)\right)^{-1}$

\STATE $p:=p-c\cdot u\cdot q$

\IF{$p\neq0$}

\STATE $p:=p\cdot\left(\lc\left(p\right)\right)^{-1}$

\ENDIF

\IF{$u\cdot\sig\left(j\right)\siglt\sig\left(k\right)$}

\STATE $\lp_{k}:=\left(\sig\left(k\right),p\right)$

\RETURN $\emptyset,\left\{ k\right\} $

\ELSE

\STATE Append $\left(u\cdot\sig\left(j\right),p\right)$ to $\lp$\label{line: append new poly (top-red)}

\STATE $\addrule\left(u\cdot\sig\left(j\right),\#\lp\right)$\label{line: add new rule (top-red)}

\RETURN $\emptyset,\left\{ k,\#\lp\right\} $

\ENDIF

\ENDBODY

\end{algorithmic}
\end{algorithm}
\begin{algorithm}
\caption{\label{alg:Find_Reductor}$\findreductor$}

\begin{algorithmic}[1]

\GLOBALS $\ordlt$

\INPUTS 

\STATE $k$, the index of a labeled polynomial

\STATE $\prevbasis\subset\mathbb{N}$, indices in $\lp$ of a Gr\"obner
basis with respect to $\ordlt$ of $\left(f_{1},f_{2},\ldots,f_{i-1}\right)$

\STATE $\currbasis\subset\mathbb{N}$, indices in $\lp$ of a list
of generators of the ideal of $\left(f_{1},f_{2},\ldots,f_{i}\right)$

\ENDINPUTS

\OUTPUTS

\STATE $J$, where $J=\left\{ j\right\} $ if $j\in\currbasis$ and
$\poly\left(k\right)$ is \emph{safely} top-reducible by $\poly\left(j\right)$;

\STATE \algindent otherwise $J=\emptyset$

\ENDOUTPUTS

\BODY

\STATE $t:=\lt\left(\poly\left(k\right)\right)$

\FOR{$j\in\currbasis$}

\STATE $t'=\lt\left(\poly\left(j\right)\right)$

\IF{$t'\mid t$}

\STATE $u:=t/t'$

\STATE $\sigvar{j}:=\sig\left(j\right)$

\IF{$u\cdot\sig\left(j\right)\neq\sig\left(k\right)$ \BAND\ \NOT\ \label{line: FindReductor test}$\rewritable\left(u,j\right)$
\BAND\ $u\cdot\multiplier_{j}$ is not top-reducible by $\prevbasis$}

\RETURN $\left\{ j\right\} $

\ENDIF

\ENDIF

\ENDFOR

\RETURN$\emptyset$

\ENDBODY

\end{algorithmic}
\end{algorithm}

\begin{itemize}
\item Performs a complete (normal form) reduction of $\poly\left(k\right)$
by the previous Gr\"obner basis.
\item Invokes $\topreduction$, which top-reduces $\poly\left(k\right)$
by the current set of generators, subject to the following restrictions.

\begin{itemize}
\item $\topreduction$ invokes $\findreductor$ to find top-reductions.
If it finds one, $\topreduction$ may act in two different ways, depending
on the signature of the top-reduction. If the signature is ``safe'',
which means ``signature-preserving'', as discussed at the end of Section~\ref{sub: props of sigs},
then an ordinary top-reduction takes place. If the signature is ``unsafe'',
then $\topreduction$ acts as if it is computing an $S$-poly\-nomial,
and thus generates a new polynomial with the new (higher) signature.
\item Some top-reductions by the current basis are forbidden by Line~\ref{line: FindReductor test}
of $\findreductor$. The practical result is that some polynomials
in the basis may not be fully top-reduced. These correspond to forbidden
$S$-poly\-nomials; compare with lines~\ref{line: critpair, PS-reducible 1}
and~\ref{line: critpair, PS-reducible 2} of $\critpair$ and line~\ref{line: spol rewritiable?}
of $\spol$.
\end{itemize}
\end{itemize}
\end{itemize}
The remaining subalgorithms record and analyze information used by
$\critpair$ and $\spol$ to discard useless pairs:%
\begin{algorithm}
\caption{\label{alg:Add_Rule}$\addrule$}

\begin{algorithmic}[1]

\GLOBALS $\lp$, $\rules$

\INPUTS

\STATE $\sigformat{\multiplier}{\indexvar}$, the signature of $\lp_{k}$

\STATE $k$, the index of a labeled polynomial in $\lp$ (or 0, for
a phantom labeled polynomial)

\ENDINPUTS

\BODY

\STATE Append $\left(\multiplier,k\right)$ to $\rules_{\indexvar}$

\RETURN{}

\ENDBODY

\end{algorithmic}
\end{algorithm}

\begin{itemize}
\item $\addrule$ is invoked whenever $\spol$ or $\altreduction$ generates
a new polynomial, and records information about that polynomial.
\item $\rewritable$ and $\findrewriting$ determine when an $S$-poly\-nomial
is re\-writ\-able.
\end{itemize}

\subsection{\label{sub:Signatures-and-labeled}Signatures and Labeled Polynomials
in F5}

The first major difference between F5 and traditional algorithms to
compute a Gr\"obner basis is the additional record keeping of ``signatures''.%
\begin{algorithm}
\caption{\label{alg:Is_Rewritable}$\rewritable$}

\begin{algorithmic}[1]

\INPUTS

\STATE $u$, a power product

\STATE $k$, the index of a labeled polynomial in $\lp$

\ENDINPUTS

\OUTPUTS

\STATE \TRUE\  if $u\cdot\sig\left(k\right)$ is rewritable (see
$\findrewriting$)

\ENDOUTPUTS

\BODY

\STATE $j:=\findrewriting\left(u,k\right)$

\RETURN $j\neq k$

\ENDBODY 

\end{algorithmic}
\end{algorithm}
\begin{algorithm}
\caption{\label{alg:Find_Rewriting}$\findrewriting$}

\begin{algorithmic}[1]

\GLOBALS $\rules$

\INPUTS 

\STATE $u$, a power product

\STATE $k$, the index of a labeled polynomial in $\lp$

\ENDINPUTS

\OUTPUTS

\STATE $j$, the index of a labeled polynomial in $\lp$ such that
if $\sigvar{j}=\sig\left(j\right)$\\
\algindent and $\sigvar{j}=\sig\left(k\right)$, then $\indexvar_{j}=\indexvar_{k}$
and $\multiplier_{j}\mid u\cdot\multiplier_{k}$\\
\algindent and $\lp_{j}$ was added to $\rules_{\nu_{k}}$ most
recently.

\ENDOUTPUTS

\BODY

\STATE $\sigformat{\multiplier_{k}}{\indexvar}:=\sig\left(k\right)$

\STATE $\ctr:=\#\rules_{\indexvar}$

\WHILE{$\ctr>0$}

\STATE $\left(\multiplier_{j},j\right):=\rules_{\indexvar,\ctr}$

\IF{$\multiplier_{j}\mid u\cdot\multiplier_{k}$}

\RETURN{$j$}

\ENDIF

\STATE $\ctr:=\ctr-1$

\ENDWHILE

\RETURN{$k$}

\ENDBODY

\end{algorithmic}
\end{algorithm}

\begin{defn}
\label{def:admissible}\label{def:signature}Let $M\in\mathbb{N}$,
$G=\left(g_{1},\ldots,g_{M}\right)\in\polyring^{M}$, and $p\in\polyring$.
We say that $\left(\multiplier,\indexvar\right)\in\mathbb{M}\times\mathbb{N}$
is a \emph{signature of $p$ with respect to} $G$
if $p$ has an $G$-re\-pre\-sen\-ta\-tion $\mathbf{h}$ such that
\begin{itemize}
\item $h_{\indexvar+1}=h_{\indexvar+2}=\cdots=h_{M}=0$; and
\item $h_\indexvar\neq 0$ and $\multiplier=\lt\left(h_{\indexvar}\right)$.
\end{itemize}
\noindent We omit the phrase ``with respect to $G$'' when it is clear
from context, and let $\multiplier\basisvar_{\indexvar}$ be a shorthand
for $\left(\multiplier,\indexvar\right)$. We also say that $\mathbf{h}$
is \emph{a}\textbf{ }\emph{$G$-representation of $p$ corresponding
to $\multiplier\basisvar_{\indexvar}$}. We call $\indexvar$ the
\emph{index}.

We also define the \emph{zero signature} $\zerosig$ of the zero polynomial
$0g_{1}+0g_{2}+\cdots+0g_{M}$.

The \emph{labeled polynomial} $\lp_{k}=\left(\sig\left(k\right),\poly\left(k\right)\right)$
is \emph{admissible with respect to $G$} if $\sig\left(k\right)$
is a signature of $\poly\left(k\right)$ with respect to $G$. Again,
we omit the phrase ``with respect to $G$'' when it is clear from
context.\end{defn}
\begin{rem*}
Our definitions of a signature differ from Faug\`ere's in several
respects:
\begin{itemize}
\item The first is minor: we use $f_{\indexvar+1}=\cdots=f_{m}=0$ whereas
Faug\`ere uses $f_{1}=\cdots=f_{\indexvar-1}=0$. The present version
simplifies considerably the description and implementation of F5C.
\item Faug\`ere uses $\left(\basisvar_{1},\ldots,\basisvar_{m}\right)$
as the basis for the $\polyring$-module $\polyring^{m}$ where $m$
is fixed; in F5C $m$ usually increases.
\item Faug\`ere's definition admits only one unique signature per polynomial,
determined by a minimality criterion. Our version allows a polynomial
to have many signatures; we refer to Faug\`ere's signature as the
\emph{minimal} signature of a polynomial. The change is motivated
by a desire to reflect the algorithm's behavior; for many inputs,
F5 does not always assign the minimal signature to a polynomial.
\item We introduce a zero signature.
\end{itemize}
\end{rem*}
The algorithm's behavior depends crucially on the assumption that
all the elements of $\lp$ are admissible. We show that the algorithm
satisfies this property in Proposition~\ref{pro:props of sigs}.
\begin{example}
\label{exa:signatures}Suppose that $F=\left(xy+x,y^{2}-1\right)$.
Then $\left(\basisvar_{1},f_{1}\right)$ is admissible with respect
to $F$. So is $\left(x\basisvar_{2},f_{1}\right)$, since $f_{1}=yf_{1}-xf_{2}$.\closer
\end{example}
It will be convenient at times to multiply monomials to signatures;
thus for any monomial $u$ and any $k\in\left\{ 1,\ldots,\#\lp\right\} $
we write the \emph{product of $u$ and $\sig\left(k\right)=\multiplier\basisvar_{\indexvar}$}
as\[
u\sig\left(k\right)=u\cdot\multiplier\basisvar_{\indexvar}=\left(u\multiplier\right)\basisvar_{\indexvar}.\]
If $\multiplier\basisvar_{\indexvar}$ is a signature of a polynomial
$p$, then the product of $u$ and $\multiplier\basisvar_{\indexvar}$
is a signature of $up$. For more properties of signatures, see Proposition~\ref{pro:props of sigs}
in Section~\ref{sub: props of sigs}.

We now generalize the ordering $\ordlt$ to an ordering on signatures.
\begin{defn}
Let $\mathcal{S}$ be the set of all possible signatures with respect
to $F$. Define a relation $\siglt$ on $\mathcal{S}$ in the following
way: for all monomials $\multiplier,\othermultiplier\in\mathbb{M}$
\begin{itemize}
\item $\zerosig$ is smaller than any other signature, and
\item for all $i,j\in\mathbb{N}$ $\othermultiplier\basisvar_{i}\siglt\multiplier\basisvar_{j}$
iff

\begin{itemize}
\item $i<j$, or
\item $i=j$ and $\othermultiplier\ordlt\multiplier$.
\end{itemize}
\end{itemize}
\end{defn}
\noindent It is clear that $\siglt$ is a well-ordering on $\mathcal{S}$,
which implies that every polynomial has a minimal signature.
\begin{example}
\label{exa:minimal signature}In Example~\ref{exa:signatures}, $\basisvar_{1}$
is the minimal signature of $f_{1}$ with respect to $F$.\closer
\end{example}

\section{\label{sec:Modification}F5C: F5 Computing with reduced Gr\"obner
bases}

It turns out that F5 often generates many ``redundant'' polynomials.
For the purposes of this discussion, a \emph{redundant polynomial
in a Gr\"obner basis} $\reducedbasis$ is a polynomial $p\in\reducedbasis$
whose head monomial is divisible by the head monomial of some $q\in\reducedbasis\backslash\left\{ p\right\} $.
It is obvious from (B) of Theorem~\ref{thm:equiv cond for GB} that
$p$ is unnecessary for the Gr\"obner basis property, and can be
discarded. In the Example given in \citep{Fau02Corrected} $r_{10}$,
which has head monomial $y^{6}t^{2}$, is a redundant polynomial because
of $r_{8}$, which has head monomial $y^{5}t^{2}$.

Why does this happen? A glance at line~\ref{line: FindReductor test}
of $\findreductor$ reveals that some top-reductions are forbidden!
Thus, despite the fact that it is often much, much faster than other
algorithms, F5 still generates many redundant polynomials. Paradoxically,
we cannot discard such polynomials safely before the algorithm has
computed a Gr\"obner basis, because the unnecessary polynomials are
marked with signatures that are necessary for the algorithm's stability
and correctness.

\subsection{\label{sub: F5C}Introducing F5C}

Stegers introduces a limited use of reduced Gr\"obner bases to F5:
variant F5R top-reduces by the polynomials of a reduced basis, but
continues to compute critical pairs and $S$-poly\-nomials with the
polynomials of the unreduced basis. One can implement this relatively
easily by changing line~\ref{line: definition of B} of  $\basisoriginal$
to

\begin{center}
\ref{line: definition of B}\algindent Let $\reducedbasis$ be the
interreduction of $\left\{ \poly\left(\lambda\right):\;\lambda\in\prevbasis\right\} $
\par\end{center}

\noindent When we say ``interreduction'', we also mean to multiply
so that the head coefficient is unity; thus $B$ is the unique reduced
Gr\"obner basis of $\ideal{F_{i}}$. Subsequently,  $\reduction$
will reduce $\poly\left(k\right)$ completely by the interreduced
$\reducedbasis$; this does not affect the algorithm's correctness
because the signature of every polynomial in $\ideal{\reducedbasis}$
is smaller than the signature of any polynomial generated with $f_{i}$.

Why does F5R only top-reduce by the reduced basis,
but not compute critical pairs and $S$-polynomials using
the reduced basis? The algorithm needs signatures and polynomials
to correspond, but the signatures of the polynomials of $\reducedbasis$
\emph{are unknown}. Merely replacing the polynomials indexed by $\prevbasis$
to those of $\reducedbasis$ would render most polynomials inadmissible.
The rewritings stored in $\rules$ would no longer correspond to the
signatures of $S$-poly\-nomials, so $\rewritable$ would reject
some $S$-poly\-nomials wrongly, and would fail to reject some $S$-poly\-nomials
when it should.

Can we get around this? In fact, we can: modify the lists $\lp$ and
$\rules$ so that the polynomials of $\reducedbasis$ are\emph{ }admissible,
and the rewrite rules valid, with respect to $\ideal{\reducedbasis}=\ideal{F_{i}}$.
Suppose that  $\partialbasis$ has terminated with value $\prevbasis$
in $\basisoriginal$. As in F5R, modify Line~\ref{line: definition of B}
of $\basisoriginal$ to interreduce $\left\{ \poly\left(\lambda\right):\;\lambda\in\prevbasis\right\} $
and obtain the reduced Gr\"obner basis $\reducedbasis$. The next
stage of the algorithm requires the computation of a Gr\"obner basis
of $\ideal{F_{i+1}}$. Certainly $\ideal{F_{i+1}}=\ideal{\reducedbasis\cup\left\{ f_{i+1}\right\} }$.
Reset $\lp$ and $\rules$, then create new lists to reflect the signatures
and rewritings for the corresponding \emph{$\reducedbasis$-}re\-pre\-sen\-ta\-tion:
\begin{itemize}
\item $\lp:=\left(\left(\basisvar_{j},\reducedbasis_{j}\right)\right)_{j=1}^{\#\reducedbasis}$;
and
\item for each $j=2,\ldots,\#\reducedbasis$ and for each $k=1,\ldots,j-1$
set $\rules_{j}:=\left(\ltmul{p,q},0\right)_{k=1}^{j-1}$ where $p=\reducedbasis_{j}$
and $q=\reducedbasis_{k}$.
\end{itemize}
The first statement assigns signatures appropriate for the module
$\polyring^{\#F'}$; the second re-creates the list of rewritings
to reflect that the $S$-poly\-nomials of $\reducedbasis$ all reduce
to zero. The redirection is to a non-existent polynomial $\lp_{0}$,
which serves as a convenient, fictional \emph{phantom polynomial};
one might say $\sig\left(0\right)=\mathbf{0}$ and $\poly\left(0\right)=0$. This reconstruction of $\lp$
and $\rules$ allows the algorithm to avoid needless reductions. (It
turns out that the reconstruction of $\rules$ is unnecessary. However,
this is not obvious, so we leave the step in for the time being, and
discuss this in Section~\ref{sec: Correctness}.) We have now rewritten
the original problem in an equivalent form, based on new information.

Although we address correctness in Section~\ref{sec: Correctness},
let us consider for a moment the intuitive reason that this phantom
polynomial $\lp_{0}$ poses no difficulty to correctness. In the original
F5 algorithm, every $S$-poly\-nomial generates a new polynomial
in $\lp$ and a corresponding rule in $\rules$. (See lines~\ref{line: if nonzero}
and~\ref{line: if nonzero add} of $\spol$, lines~\ref{line: if poly(k) =00003D 0}--\ref{line: zero reduction, return 0,0}
of $\topreduction$, and lines~\ref{line: for to add new polys}
and~\ref{line: add new polys} of $\reduction$.) If $\lp_{k}$ reduces
to zero for some $k$, then $k$ is not added to $\currbasis$, but
the rewrite rule $\left(\sig\left(k\right),k\right)$ remains in $\rules$.
Thus the algorithm never uses $\poly\left(k\right)$ again; however,
it uses $\sig\left(k\right)$ to avoid computing other polynomials
with the same signature. The change we propose has the same effect
on $S$-poly\-nomials of $\reducedbasis$: we know \emph{a priori
}that they reduce to zero. We could add a large number of entries
$\left(\sig\left(k\right),0\right)$ to $\lp$, but since the algorithm
never uses them we would merely waste space. Instead, we redirect
the signature $\sig\left(k\right)$ to a phantom polynomial $\lp_{0}$,
which like $\lp_{k}$ is never in fact used.

We call the resulting algorithm F5C, and summarize the modifications
in the pseudocode of Algorithms~\ref{alg:F5C} and \ref{alg:Setup_Reduced_Basis};
the first replaces Algorithm~\ref{alg:F5} entirely. We have separated
most of the modification of $\basisoriginal$ into $\createreducedbasis$,
a separate subalgorithm invoked by $\basis$, the replacement for
$\basisoriginal$.

\begin{algorithm}
\caption{\label{alg:F5C}$\basis$}

\begin{algorithmic}[1]

\GLOBALS$\lp$, $\rules$, $\ordlt$

\INPUTS

\STATE$F=\left(f_{1},f_{2},\ldots,f_{m}\right)\in\polyring^{m}$
(homogeneous)

\STATE$<$, an admissible ordering

\ENDINPUTS

\OUTPUTS

\STATE a Gr\"obner basis of $F$ with respect to $<$

\ENDOUTPUTS

\BODY

\STATE$\ordlt:=<$

\STATE Sort $F$ by increasing total degree, breaking ties by increasing
leading monomial

\STATE $\rules:=\newlist{\newlist{}}$

\STATE $\lp:=\newlist{}$

\STATE Append $\left(\sigformat{}{1},f_{1}\cdot\lc\left(f_{1}\right)^{-1}\right)$
to $\lp$

\STATE $\prevbasis=\left\{ 1\right\} $

\STATE $\reducedbasis=\left\{ f_{1}\right\} $

\STATE $i:=2$

\WHILE{ $i\leq m$}

\STATE Append $\left(\basisvar_{\#\lp+1},f_{i}\cdot\lc\left(f_{i}\right)^{-1}\right)$
to $\lp$

\STATE $\currbasis:=\partialbasis\left(\#\lp,\reducedbasis,\prevbasis\right)$

\IF{$\exists\lambda\in\currbasis$ such that $\poly\left(\lambda\right)=1$}

\RETURN $\left\{ 1\right\} $

\ENDIF

\COMMENT{The only change to $\basisoriginal$ is the addition of
this line}

\STATE $\prevbasis:=\createreducedbasis\left(\currbasis\right)$

\STATE $\reducedbasis:=\left\{ \poly\left(\lambda\right):\;\lambda\in\prevbasis\right\} $

\STATE $i:=i+1$

\ENDWHILE

\RETURN \textbf{$\reducedbasis$}

\ENDBODY

\end{algorithmic}
\end{algorithm}

\begin{algorithm}
\caption{\label{alg:Setup_Reduced_Basis}$\createreducedbasis$}

\begin{algorithmic}[1]

\GLOBALS $\lp$, $\rules$, $\ordlt$

\algindent (modifies $\lp$ and $\rules$)

\INPUTS 

\STATE $\prevbasis$, a list of indices of polynomials in $\lp$
that correspond to a Gr\"obner basis of $\left(f_{1},\ldots,f_{i}\right)$

\ENDINPUTS 

\OUTPUTS 

\STATE $\currbasis\subset\mathbb{N}$, indices of polynomials in
$\lp$ that correspond to a \emph{reduced} Gr\"obner basis of $\left(f_{1},\ldots,f_{i}\right)$

\ENDOUTPUTS 

\BODY 

\STATE Let $B$ be the interreduction of $\left\{ \poly\left(k\right):\; k\in\prevbasis\right\} $\label{line: compute reduced basis}

\STATE $\currbasis:=\left\{ j\right\} _{j=1}^{\#\reducedbasis}$

\STATE $\lp:=\newlist{\left\{ \left(\basisvar_{j},\reducedbasis_{j}\right)\right\} _{j=1}^{\#\reducedbasis}}$\label{line: new signature}

\comment{Lemma \ref{thm: can skip rewrite rules for B} implies that
lines \ref{line: start to create new rules}--\ref{line: finish creating new rules}
are unnecessary}

\comment{All the $S$-polynomials of $\reducedbasis$ reduce to zero;
document this}

\STATE $\rules=\newlist{\left\{ \newlist{}\right\} _{j=1}^{\#\reducedbasis}}$\label{line: start to create new rules}

\FOR{$j:=1$ \TO\ \textbf{ }$\#\reducedbasis-1$}

\STATE $t:=\lt\left(\reducedbasis_{j}\right)$

\FOR{$k:=j+1$ \TO\  $\#\reducedbasis$}

\STATE $u:=\lcm\left(t,\lt\left(\reducedbasis_{k}\right)\right)/\lt\left(\reducedbasis_{k}\right)$

\STATE $\addrule\left(\sigformat{u}{k},0\right)$\label{line: new rule}\label{line: finish creating new rules}

\ENDFOR

\ENDFOR

\RETURN $\currbasis$

\ENDBODY

\end{algorithmic}
\end{algorithm}

\subsection{Experimental results}

One way to compare the three variants would be to measure the absolute
timings when computing various benchmark systems. By this metric,
F5R generally outperforms F5, and F5C generally outperforms F5R: the
exceptions are all toy systems, where the overhead of repeated interreduction
and $\createreducedbasis$ outweigh the benefit of using a reduced
Gr\"obner basis.%
\begin{table}
\begin{centering}
\begin{tabular}{|c|c|c|c|c|c|}
\hline 
system & F5 (sec) & F5R (sec) & F5C (sec) & F5R/F5 & F5C/F5\tabularnewline
\hline
\hline 
Katsura-7 & 6.60 & 5.09 & 4.23 & 0.77 & 0.64\tabularnewline
\hline 
Katsura-8 & 111.05 & 52.22 & 43.88 & 0.47 & 0.40\tabularnewline
\hline 
Katsura-9 & 5577 & 1421 & 1228 & 0.25 & 0.22\tabularnewline
\hline 
Cyclic-6 & 3.91 & 3.88 & 3.41 & 0.99 & 0.87\tabularnewline
\hline 
Cyclic-7 & 1182 & 505 & 381 & 0.43 & 0.32\tabularnewline
\hline 
Cyclic-8 & >4 days & 231455 & 188497 & N/A & N/A\tabularnewline
\hline
\end{tabular}
\par\end{centering}

\caption{\label{tab: results-Sage}Ratios of timings in the Sage (Python) implementation}

All timings obtained using the \texttt{cputime()} function in a Python
implementation in Sage 3.2.1, on a computer with a 2.66GHz Intel Core
2 Quad (Q9450) running Ubuntu Linux with~3GB RAM. The ground field
has characteristic 32003. {*}Computation of Cyclic-8 in F5 has not
terminated on the sixth day of computation, when this draft was committed.
On other computers, the timing was comparable.
\end{table}
\begin{table}
\begin{centering}
\begin{tabular}{|c|c|c|c|c|c|}
\hline 
system & F5 (sec) & F5R (sec) & F5C (sec) & F5R/F5 & F5C/F5\tabularnewline
\hline
\hline 
Katsura-7 & 0.30 & 0.34 & 0.31 & 1.13 & 1.03\tabularnewline
\hline 
Katsura-8 & 4.05 & 4.41 & 3.33 & 1.09 & 0.82\tabularnewline
\hline 
Katsura-9 & 127.14 & 142.81 & 82.48 & 1.12 & 0.65\tabularnewline
\hline 
Schrans-Troost & 25.43 & 21.74 & 21.43 & 0.85 & 0.84\tabularnewline
\hline 
F633 & 0.34 & 0.40 & 0.30 & 1.18 & 0.88\tabularnewline
\hline 
F744 & 1252 & 1132 & 1075 & 0.90 & 0.86\tabularnewline
\hline 
Cyclic-6 & .04 & .03 & .03 & 0.75 & 0.75\tabularnewline
\hline 
Cyclic-7 & 6.5 & 5.39 & 4.35 & 0.83 & 0.67\tabularnewline
\hline 
Cyclic-8 & 3233 & 3101 & 2154 & 0.96 & 0.67\tabularnewline
\hline
\end{tabular}
\par\end{centering}

\caption{\label{tab: results-Singular}Timings for the (compiled) \textsc{Singular}
implementations}

Average of four timings obtained from the \texttt{getTimer()} function
in a modified \textsc{Singular} 3-1-0 kernel, on a computer with a
3.16GHz Intel Xeon (X5460) running Gentoo Linux with~64GB RAM. The
ground field has characteristic 32003.
\end{table}
 Tables~\ref{tab: results-Sage} and~\ref{tab: results-Singular}
give timings and ratios for the variants in two different implementations.
\begin{itemize}
\item Table \ref{tab: results-Sage} gives the results from an implementation
written in Python for the Sage computer algebra system, version 3.4.
Sage is built on several other systems, one of which is \textsc{Singular}
3-0-4. Sage calls \textsc{Singular} to perform certain operations,
so some parts of the implementation run in compiled code, but most
of the algorithm is otherwise implemented in Python. For example,
Line~\ref{line: normal form} of $\reduction$ (reduction by the
previous basis) is handed off to \textsc{Singular}, while the implementation
of \topreduction\  is nearly entirely Python.
\item Table \ref{tab: results-Singular} gives the results from a compiled
\textsc{Singular} implementation built on the \textsc{Singular}~3-1
kernel. This implementation is unsurprisingly much, much faster than
the Sage implementation. Nevertheless, the implementation is still
a work in progress, lacking a large number of optimizations. For example,
so far polynomials are represented by geobuckets~\citep{Yap2000};
the eventual goal is to implement the F4-style reduction that Faug\`ere
advises for efficiency~\citep{Fau99,Fau02Corrected}.\end{itemize}
\begin{rem*}
This \textsc{Singular} implementation has one major difference from
the pseudocode given: its implementation of $\topreduction$ performs
safe reductions of non-leading monomials as well as of the leading
monomials. This helps explain why there seems to be no benefit to
F5R, unlike the Sage implementation. Another factor is that top-reduction
in Sage is performed by interpreted Python code, whereas tail reductions
are performed by the compiled \textsc{Singular} library to which Sage
links. Thus, the penalty for interreduction, relative to top-reduction,
is much lower in Sage, to the benefit of F5R.
\end{rem*}
Timings alone are an unsatisfactory metric for this comparison. They
depend heavily on the efficiency of hidden algorithms, such as the
choice of polynomial representation (lists, buckets, sparse matrices).
It is well-known that the most time-consuming part by far of any non-trivial
Gr\"obner basis computation consists in the reduction operations:
top-reduction, inter-reduction, and computing normal forms. This remains
true for F5, with the additional wrinkle that, as mentioned before,
F5 generally computes many more polynomials than are necessary for
the Gr\"obner basis. Thus a more accurate comparison between the
three variants would consider
\begin{itemize}
\item the number of critical pairs considered,
\item the number of polynomials generated, and
\item the number of reduction operations performed.
\end{itemize}
We present a few examples with benchmark systems in Tables~\ref{tab: comparison by reductions performed}--\ref{tab: Katsura-9 in F5C},
generated from the prototype implementation in Sage.%
\begin{table}
\begin{centering}
\begin{tabular}{|c|c|c|c|}
\hline 
system & reductions in F5 & reductions in F5R & reductions in F5C\tabularnewline
\hline
\hline 
Katsura-4 & 774 & 289 & 222\tabularnewline
\hline 
Katsura-5 & 14597 & 5355 & 3985\tabularnewline
\hline 
Katsura-6 & 1029614 & 77756 & 58082\tabularnewline
\hline 
Cyclic-5 & 510 & 506 & 446\tabularnewline
\hline 
Cyclic-6 & 41333 & 23780 & 14167\tabularnewline
\hline
\end{tabular}
\par\end{centering}

\centering{}\caption{\label{tab: comparison by reductions performed}Reductions performed
by the three F5 variants over a field of characteristic 32003.}

\end{table}
 In each case, the number of reductions performed by F5C remains substantially
lower than the number performed by F5R, which is itself drastically
lower than the number performed by F5. As a reference for comparison,
we modified the toy implementation of the Gebauer-M\"oller algorithm
that is included with the Sage computer algebra system to count all
the reduction operations \citep{GM88}; it performed more than 1,500,000
reductions to compute Cyclic-6. The table shows that F5 performed
approximately 2.4\% of that number, while F5C performed approximately
0.7\% of that number.

In general, F5 and F5R will compute the same number of critical pairs
and polynomials, because they are using the same values of $\prevbasis$.
Top-reducing by a reduced Gr\"obner basis eliminates the vast majority
of reductions, but in F5R $\prevbasis$ still indexes polynomials
whose monomials are reducible by other polynomials, \emph{including
head monomials!} As a consequence, F5R cannot consider fewer critical
pairs or generate fewer polynomials than F5. By contrast, F5C has
discarded from $\prevbasis$ polynomials with redundant head monomials,
and has eliminated reducible lower order monomials. Correspondingly,
there is less work to do.
\begin{example}
\label{exa: Katsura-9 performance}In the Katsura-9 system for F5
and F5R, each pass through the \verb|while| loop of  $\partialbasis$
generates the internal data shown in Table~\ref{tab: Katsura-9 in F5,F5R}.
For F5C, each pass through the \verb|while| loop of $\partialbasisc$
generates the internal data shown in Table~\ref{tab: Katsura-9 in F5C}.
For each $i$, F5R and F5C both compute $\reducedbasis$, the unique
reduced Gr\"obner basis of $F_{i}$. This significantly speeds up
top-reduction, but F5C replaces $\lp$ with labeled polynomials for
$\reducedbasis$. The consequence is that $\prevbasis$ contains fewer
elements, leading $\partialbasisc$ to generate fewer critical pairs,
and hence fewer polynomials for $\currbasis$. Similar behavior occurs
in other large systems.\closer%
\begin{table}
\begin{centering}
\begin{tabular}{|c|c|c|c|}
\hline 
$i$ & $\#\currbasis$ & $\max\left\{ d\right\} $ & $\max\left\{ \#\critpairs_{d}\right\} $\tabularnewline
\hline
\hline 
2 & 2 & N/A & N/A\tabularnewline
\hline 
3 & 4 & 3 & $\#\critpairs_{3}=1$\tabularnewline
\hline 
4 & 8 & 4 & $\#\critpairs_{3}=2$\tabularnewline
\hline 
5 & 16 & 6 & $\#\critpairs_{4}=\#\critpairs_{5}=4$\tabularnewline
\hline 
6 & 32 & 6 & $\#\critpairs_{4}=8$\tabularnewline
\hline 
7 & 60 & 10 & $\#\critpairs_{5}=17$\tabularnewline
\hline 
8 & 132 & 11 & $\#\critpairs_{6}=29$\tabularnewline
\hline 
9 & 524 & 16 & $\#\critpairs_{8}=89$\tabularnewline
\hline 
10 & 1165 & 13 & $\#\critpairs_{8}=276$\tabularnewline
\hline
\end{tabular}
\par\end{centering}

\caption{\label{tab: Katsura-9 in F5,F5R}Internal data of $\partialbasis$
in both F5 and F5R while computing Katsura-9.}

\end{table}
\begin{table}
\begin{centering}
\begin{tabular}{|c|c|c|c|}
\hline 
$i$ & $\#\currbasis$ & $\max\left\{ d\right\} $ & $\max\left\{ \#\critpairs_{d}\right\} $\tabularnewline
\hline
\hline 
2 & 2 & N/A & N/A\tabularnewline
\hline 
3 & 4 & 3 & $\#\critpairs_{3}=1$\tabularnewline
\hline 
4 & 8 & 4 & $\#\critpairs_{3}=2$\tabularnewline
\hline 
5 & 15 & 6 & $\#\critpairs_{3}=\#\critpairs_{4}=4$\tabularnewline
\hline 
6 & 29 & 6 & $\#\critpairs_{4}=\#\critpairs_{6}=6$\tabularnewline
\hline 
7 & 51 & 10 & $\#\critpairs_{5}=12$\tabularnewline
\hline 
8 & 109 & 11 & $\#\critpairs_{6}=29$\tabularnewline
\hline 
9 & 472 & 16 & $\#\critpairs_{8}=71$\tabularnewline
\hline 
10 & 778 & 13 & $\#\critpairs_{8}=89$\tabularnewline
\hline
\end{tabular}
\par\end{centering}

\caption{\label{tab: Katsura-9 in F5C}Internal data of $\partialbasisc$ in
F5C while computing Katsura-9.}

\end{table}

\end{example}

\section{\label{sec: F5 correct}Correctness of the output of F5 and F5C}

In this section we prove that if F5 and F5C terminate, then their
output is correct. Seeing that Faug\`ere has already proved the correctness
of F5, why do we include a new proof? First, we rely on certain aspects
of the proof to explain the modifications that led to F5C, so it is
convenient to re-present a proof here. Another reason is to present
a new generalization of Faug\`ere's characterization of a Gr\"obner
basis; although it is not necessary for F5C, the new characterization
is interesting enough to describe here.
\begin{rem*}
We do not address the details of termination, nor will we even assert
that the algorithms \emph{do} terminate, but in practice we have not
encountered any systems that do not terminate in F5.

Having said that, we would like to address an issue with which some
readers may be familiar. The Magma source code of \citep{Stegers2006}
implements F5R. This code is publicly available, and contains an example
system in the file \texttt{nonTerminatingExample.mag}. As the reader
might expect from the name, this system causes an infinite loop when
given as input to the source code. Roger Dellaca, Justin Gash, and
John Perry traced this loop to an error in $\topreduction$. (Lines~\ref{line: append new poly (top-red)}
and~\ref{line: add new rule (top-red)} were not implemented, which
sabotages the record-keeping of $\rules$.) The corrected Magma code
terminates with the Gr\"obner basis of that system.
\end{rem*}

\subsection{\label{sub: props of sigs}Properties of signatures}

The primary tool in F5 is the signature of a polynomial (Definition~\ref{def:signature}).
The following properties of signatures explain certain choices made
by the algorithm.
\begin{prop}
\label{pro:props of sigs}Let $p,q\in\polyring$, $\multiplier,\othermultiplier,u,v\in\mathbb{M}$,
and $\indexvar,\otherindexvar\in\left\{ 1,2,\ldots,M\right\} $. Suppose
that $\multiplier\basisvar_{\indexvar}$ and $\othermultiplier\basisvar_{\otherindexvar}$
are signatures of $p$ and $q$, respectively. Each of the following
holds:
\begin{lyxlist}{(B)}
\item [{(A)}] $\left(u\multiplier\right)\basisvar_{\indexvar}$ is a signature
of $up$.
\item [{(B)}] If $u\multiplier\basisvar_{\indexvar}\siggt\othermultiplier\basisvar_{\otherindexvar}$,
then $\left(u\multiplier\right)\basisvar_{\indexvar}$ is a signature
of $up\pm vq$.
\item [{(C)}] If $\left(\ltmul{p,q}\multiplier\right)\basisvar_{\indexvar}\siggt\left(\ltmul{q,p}\othermultiplier\right)\basisvar_{\otherindexvar}$,
then $\left(\ltmul{p,q}\multiplier\right)\basisvar_{\indexvar}$ is
a signature of $\SPol\left(p,q\right)$.
\end{lyxlist}
\end{prop}
\noindent The proof is straightforward, so we omit it.
\begin{defn}
Let $u,v\in\termset$ and $j,k\in\{1,\ldots,\#\lp\}$.
We say that \emph{the natural signature of $u\poly(j)$ from $\sig(j)$ with respect to $F$}
is the signature deduced by Proposition~\ref{pro:props of sigs}(A).
We usually omit ``from $\sig(j)$ with respect to $F$'' since it is clear from context.
We similarly define the \emph{natural signature of $u\poly(j)\pm v\poly(k)$}
\emph{(from $\sig(j)$ and $\sig(k)$)} from (B) and
\emph{the natural siganture of $\SPol\left(p,q\right)$ (from $\sig(j)$ and $\sig(k)$)}
from (C). If the hypotheses of (B) and (C) are unsatisfied,
then the natural signature is undefined.
\end{defn}

The following proposition implies that the labeled polynomials of
$\lp$ are admissible with respect to the input at every moment during
the algorithm's execution.
\begin{prop}
\label{pro:red to pr over smaller sigs}Each of the following holds.
\begin{lyxlist}{(B)}
\item [{(A)}] For every $k\in\left\{ 1,2,\ldots,\#\lp\right\} $, $\sig\left(k\right)$
is the natural signature of $\poly\left(k\right)$ with respect to $F$ when
$\lp_{k}$ is defined in Line~\ref{line: create new poly from S-poly}
of $\spol$ and Line~\ref{line: append new poly (top-red)} of $\topreduction$.
\item [{(B)}] After the call\[
h:=\normalform\left(\poly\left(k\right),\safereducers,\ordlt\right)\]
in Line~\ref{line: normal form} of  $\reduction$, $\sig\left(k\right)$
is the natural signature of $h$ with respect to $F$.
\item [{(C)}] For all $k\in\left\{ 1,2,\ldots,\#\lp\right\} $, $\sig\left(k\right)$
remains invariant, and is the natural signature of $\poly\left(k\right)$ with
respect to $F$.
\end{lyxlist}
\end{prop}
\noindent The proof follows without difficulty from Proposition \ref{pro:props of sigs}
and inspection of the algorithms that create or modify labeled polynomials:
$\partialbasis$, $\spol$, $\reduction$, and $\topreduction$.
Top-reductions that generate new polynomials correspond to new $S$-polynomials;
they are simply ``discovered'', and generated, in a different place.
\begin{rem*}
Although $\sig\left(k\right)$ is \emph{a} signature of $\poly\left(k\right)$,
it need not be the \emph{minimal} signature of $\poly\left(k\right)$.
For example, if F5 is given the input $F=\left(xh+h^{2},yh+h^{2}\right)$
then $\spol$ computes an $S$-poly\-nomial and creates the labeled
polynomial\[
\lp_{3}=\left(x\basisvar_{2},yh^{2}-xh^{2}\right).\]
Hence $\sig\left(3\right)=x\basisvar_{2}$, but it is also true that\[
\SPol\left(f_{1},f_{2}\right)=-hf_{1}+hf_{2}.\]
Thus $h\basisvar_{2}$ is also a signature of $\poly\left(3\right)$;
in fact, it is the minimal signature.
Since $h\basisvar_{2}\siglt x\basisvar_{2}$, $x\basisvar_{2}$ is
not the minimal signature of $f_{2}$, although it is the natural signature.\end{rem*}
We can now explain what is meant by a ``safe'' top-reduction.
\begin{defn}
\label{def:sig-preserving S-rep}Let $F\in\polyring^{m}$; all signatures
are with respect to $F$. Suppose that $\multiplier\basisvar_{\indexvar}$
is the natural signature of an $S$-polynomial $S$ generated by $\poly\left(a\right)$
and $\poly\left(b\right)$, and $\mathbf{h}$ is an $S$-representation
of $S$ such that the natural signatures of the products satisfy\[
\lt\left(h_{\lambda}\right)\sig\left(\lambda\right)\siglt\multiplier\basisvar_{\indexvar}\]
for all $\forall\lambda=1,\ldots,\#\mathbf{h}$ \emph{except one},
say $\lambda'$, in which case $\lt\left(h_{\lambda'}\right)\sig\left(\lambda'\right)=\multiplier\basisvar_{\indexvar}$
and $\lambda'>a,b$. We call $\mathbf{h}$ a \emph{signature-preserving}\textbf{\emph{
}}$S$-repre\-sen\-ta\-tion.
\end{defn}
Proposition~\ref{pro:red to pr over smaller sigs} implies that top-reductions
that do not generate new polynomials create signature-preserving $S$-repre\-sen\-ta\-tions
of $S$-poly\-nomials. Top-reductions that do generate new polynomials
correspond to new $S$-poly\-nomials, and the reductions of the new
polynomials likewise correspond to signature-preserving $S$-re\-pre\-sen\-ta\-tions.
Thus, if we are at a stage of
the algorithm where $\spol$ generated $\lp_{k}$, but $\reduction$
has not yet reduced it, we say that $\reduction$ \emph{is} \emph{scheduled
to compute a signature-preserving $S$-repre\-sen\-ta\-tion}. Once
it computes the representation, we say that the algorithm \emph{has
computed a signature-preserving reduction to zero.}

\subsection{\label{sub:Rewritable-polynomials}Rewritable Polynomials}

As Faug\`ere illustrates in Section~2 of \citep{Fau02Corrected},
linear algebra suggests that two rows of the Sylvester matrix of $F$
need not be triangularized if one row has already been used in the
triangularization of another row. This carries over into the $F$-re\-pre\-sen\-ta\-tions
of $S$-poly\-nomial components, so F5 uses signatures to hunt for
such redundant components. The structure $\rules$ tracks which signatures
have already been computed.
\begin{defn}
\label{def: list of rewritings}Let $\rules$ be a list of $m$ lists
of tuples of the form $\rho=\left(\multiplier,j\right)$. We write
$\rules_{i}$ for the $i$th list in $\rules$. We say that $\rules$
is a \emph{list of rewritings}\textbf{\emph{ }}\emph{for}\textbf{\emph{
$\lp$}} if for every $i=1,\ldots,m$ and for every $\rho_{\ell}=\left(\multiplier,j\right)\in\rules_{i}$
there exist $p,q\in\polyring$ such that
\begin{enumerate}
\item $p=\poly\left(a\right)$, $q=\poly\left(b\right)$ for some $a,b\in\currbasis$;
\item $\max_{\siglt}\left\{ \ltmul{p,q}\cdot\sig\left(a\right),\ltmul{q,p}\cdot\sig\left(b\right)\right\} =\multiplier\basisvar_{i}$;
\item $j>a,b$ and the first defined value of $\poly\left(j\right)$ is
$\SPol\left(p,q\right)$;
\item there exists (or  $\reduction$ is scheduled to compute) a signature-preserving
$S$-re\-pre\-sen\-ta\-tion $\mathbf{h}$ of $\SPol\left(p,q\right)$
such that $h_{j}=1$; and
\item if $\rho_{\ell'}=\left(\othermultiplier,j'\right)\in\rules_{i}$ and
$\ell'>\ell$, then $j'>j$.
\end{enumerate}
\noindent We call $\poly\left(j\right)$ the \emph{rewriting} of $\SPol\left(p,q\right)$.\end{defn}
\begin{rem*}
When we speak of $\SPol\left(p,q\right)$, we include any unsafe top-reduction
that is computed in $\topreduction$.\end{rem*}
\begin{prop}
\label{pro: top-reds recorded in rules}Every signature-preserving
reduction by F5 of an $S$-poly\-nomial $S$ to the polynomial $p$
(where possibly $p=0$) is recorded in some $\rules_{i}$ by the entry
$\left(u\cdot\sig\left(k\right),j\right)$ where:
\begin{itemize}
\item $S=u\cdot\poly\left(k\right)-v\cdot\poly\left(\ell\right)$ for some
$\ell\in\currbasis$ and appropriate $u,v\in\mathbb{M}$;
\item $u\cdot\sig\left(k\right)\siggt v\cdot\sig\left(\ell\right)$;
\item the first defined value of $\poly\left(j\right)$ is $S$, and the
final value of $\poly\left(j\right)$ is $p$; and
\item $j>k,\ell$.
\end{itemize}
\end{prop}
The proof follows from inspection of the algorithms that create and
top-reduce poly\-nomials.
\begin{prop}
At every point during the execution of F5, the global variable $\rules$
satisfies Definition \ref{def: list of rewritings}.
\end{prop}
The proof follows from Proposition~\ref{pro: top-reds recorded in rules}
and inspection of the algorithms that create and modify $\rules$.
\begin{defn}
\label{def: rewritable}Let $j,k\in\currbasis$, $u\in\mathbb{M}$,
and $\sig\left(k\right)=\multiplier\basisvar_{\indexvar}$. At any
given point during the execution of the algorithm we say that the
polynomial multiple $u\poly\left(k\right)$ is \emph{rewritable by
$\poly\left(j\right)$}\textbf{\emph{ }}\emph{in $W=\rules_{\indexvar}$}
if
\begin{itemize}
\item $k\neq j$;
\item $\poly\left(j\right)$ is the rewriting of an $S$-poly\-nomial;
\item $\sig\left(j\right)=\multiplier'\basisvar_{\indexvar}$ and $\multiplier'\mid u\multiplier$
(note the same index $\indexvar$ as $\sig\left(k\right)$);
\item $\left(\multiplier',j\right)=W_{a}$ for some $a\in\mathbb{N}$; and
\item for any $W_{b}=\left(\multiplier'',c\right)$ such that $\multiplier''\mid u\multiplier$,
either $W_{a}=W_{b}$ or $b<a$.
\end{itemize}
We usually omit some or all of the phrase ``by $\poly\left(j\right)$
in $\rules_{\indexvar}$.'' We call $\poly\left(j\right)$ the \emph{rewriter}
of $u\poly\left(k\right)$.
\end{defn}

\begin{prop}
\label{pro: algorithm rewritable detects rewritables}Let $u\in\mathbb{M}$
and $k\in\currbasis$. The following are equivalent.
\begin{lyxlist}{(M)}
\item [{(A)}] $u\poly\left(k\right)$ is rewritable in $\rules_{\indexvar}$,
where $\sig\left(k\right)=\multiplier\basisvar_{\indexvar}$ for some
$\multiplier\in\mathbb{M}$.
\item [{(B)}] $\rewritable\left(u,k\right)$ returns \verb|true|.
\end{lyxlist}
\end{prop}
\noindent The proof follows from inspection of the algorithms that
create, inspect, and modify $\rules$.

\begin{prop}
\label{pro:rewriter has larger index than rewritable}If a polynomial
multiple $u\poly\left(k\right)$ is rewritable, then the rewriter
$\poly\left(j\right)$ satisfies $j>k$.
\end{prop}
The proof follows from Definitions~\ref{def: list of rewritings}
($j'>j$) and~\ref{def: rewritable} ($b<a$).
\begin{prop}
\label{pro: structure of rewriting}Let $k\in\currbasis$. Suppose
that a polynomial multiple $p=u\poly\left(k\right)$ is rewritable
by some $\poly\left(j\right)$ in $\rules_{\indexvar}$. If $\reduction$
terminates, then there exist $c\in\mathbb{F}$, $d\in\mathbb{M}$
and $h_{\lambda}\in\polyring$ (for each $\lambda\in\left(\currbasis\cup\completed\right)\backslash\left\{ j\right\} $)
satisfying\begin{equation}
p=cd\cdot\poly\left(j\right)+\sum_{\lambda\in\left(\currbasis\cup\completed\right)\backslash\left\{ j\right\} }h_{\lambda}\poly\left(\lambda\right)\label{eq: SPol rep after rewriting}\end{equation}
where
\begin{itemize}
\item for all $\lambda\in\left(\currbasis\cup\completed\right)\backslash\left\{ j\right\} $
if $h_{\lambda}\neq0$ then the natural signature of $\lt\left(h_{\lambda}\right)\cdot\poly\left(\lambda\right)$
is smaller than $u\sig\left(k\right)$; and
\item $u\sig\left(k\right)$ is the natural signature of $cd\cdot\poly\left(j\right)$.
\end{itemize}
\end{prop}
\begin{rem*}
It does \emph{not} necessarily follow that $\mathbf{h}$ is an $\lt\left(p\right)$-representation
of $p$. The usefulness of Proposition \ref{pro: structure of rewriting}
lies in the fact that all polynomials in~\eqref{eq: SPol rep after rewriting}
have a smaller signature than $p$ except possibly $cd\cdot\poly\left(j\right)$.
If $\poly\left(j\right)=0$ then the Proposition still holds, since
$u\sig\left(k\right)$ would be a non-minimal signature of the zero
polynomial.\end{rem*}
\begin{proof}
Assume that  $\reduction$ terminates. Let $\sig\left(k\right)=\multiplier\basisvar_{\indexvar}$.
By Definition~\ref{def:signature} there exist $q_{1},\ldots,q_{\indexvar}\in\polyring$
such that\[
p=q_{1}f_{1}+\cdots+q_{\indexvar}f_{\indexvar},\]
and $\lt\left(q_{\indexvar}\right)=\multiplier$. Let $\sig\left(j\right)=\multiplier'\basisvar_{\indexvar}$
and let $S$ be the $S$-poly\-nomial that generated $\poly\left(j\right)$.
By Definitions~\ref{def: list of rewritings} and~\ref{def: rewritable},
there exist $H_{1},\ldots,H_{\indexvar}\in\polyring$ such that\[
S=H_{1}f_{1}+\cdots+H_{\indexvar}f_{\indexvar}\]
and
\begin{itemize}
\item $\lt\left(H_{\indexvar}\right)=\multiplier'$,
\item $\multiplier'\mid u\multiplier$,
\item $\rho=\left(\multiplier',j\right)$ appears in $\rules_{\indexvar}$,
\item and $k\neq j$.
\end{itemize}
\noindent Let $\mathcal{G}=\currbasis\cup\completed$. By Definition~\ref{def: list of rewritings}
and the assumption that $\reduction$ terminates, there exists $\mathcal{H}\in\polyring^{\#\mathcal{G}}$
such that
\begin{itemize}
\item \noindent $\mathcal{H}$ is a signature-preserving $S$-re\-pre\-sen\-ta\-tion
of $S$ w.r.t. $\left\{ \poly\left(\lambda\right):\;\lambda\in\mathcal{G}\right\} $;
and
\item \noindent $\mathcal{H}_{j}=1$.
\end{itemize}
Let $d$ be a monomial such that $d\multiplier'=u\multiplier$. Thus
$d\sig\left(j\right)=u\sig\left(k\right)$. Let $\alpha=\lc\left(h_{\nu}\right)$
and $\beta=\lc\left(H_{\nu}\right)$. Note that $\beta\neq0$, since
it comes from an assigned signature. Then\begin{align}
p & =\left[\left(q_{1}f_{1}+\cdots+q_{\nu}f_{\nu}\right)-\frac{\alpha}{\beta}dS\right]+\frac{\alpha}{\beta}dS\nonumber \\
 & =\left[\sum_{\lambda=1}^{\nu}\left(q_{\lambda}-\frac{\alpha}{\beta}dH_{\lambda}\right)f_{\lambda}\right]+\left[\frac{\alpha}{\beta}d\poly\left(j\right)+\sum_{\lambda\in\mathcal{G}\backslash\left\{ j\right\} }\left(\frac{\alpha}{\beta}d\mathcal{H}_{\lambda}\right)\poly\left(\lambda\right)\right]\nonumber \\
 & =\frac{\alpha}{\beta}d\cdot\poly\left(j\right)+\sum_{\lambda\in\mathcal{G}\backslash\left\{ j\right\} }h_{\lambda}\poly\left(\lambda\right)\label{eq: finally a good exp for p}\end{align}
where\[
h_{\lambda}=\begin{cases}
q_{\lambda}-\frac{\alpha}{\beta}d\left(H_{\lambda}-\mathcal{H}_{\lambda}\right),\quad & \mbox{if }\poly\left(\lambda\right)=f_{k}\mbox{ for some }k=1,\ldots,\nu;\\
\frac{\alpha}{\beta}d\mathcal{H}_{\lambda} & \mbox{otherwise.}\end{cases}\]
Recall that\[
\lt\left(q_{\indexvar}\right)=u\multiplier=\lt\left(\frac{\alpha}{\beta}d\cdot H_{\indexvar}\right)\]
and since $\mathcal{H}$ is signature-preserving\[
\lt\left(\frac{\alpha}{\beta}d\cdot\mathcal{H}_{\lambda}\right)\sig\left(\lambda\right)\siglt d\multiplier'\basisvar_{\indexvar}=u\multiplier\basisvar_{\indexvar}\quad\forall\lambda\in\mathcal{G}\backslash\left\{ j\right\} .\]
Thus for any $\lambda\in\mathcal{G}\backslash\left\{ j\right\} $
if $h_{\lambda}\neq0$ then $\lt\left(h_{\lambda}\right)\sig\left(\lambda\right)\siglt u\sig\left(k\right)$.
Recall that $d\sig\left(j\right)=u\sig\left(k\right)$. Let $c=\alpha/\beta$;
then equation~\eqref{eq: finally a good exp for p} satisfies the
proposition.
\end{proof}
We stumbled on Lemma~\ref{lem: same sig rewritable} while trying
to resolve a question that arose in our study of the pseudocode of
\citep{Fau02Corrected} and \citep{Stegers2006}. Among the criteria
that they use to define a \emph{normalized critical pair}, they mention
that the signatures of the corresponding polynomial multiples must
be different. However, their pseudocodes for $\critpair$ do not check
for this! This suggests that they risk generating at least a few critical
pairs that are not normalized, but we have found that this does not
occur in practice. Why not?
\begin{lem}
\label{lem: same sig rewritable}Let $k,\ell\in\currbasis$ with $k>\ell$.
Let $p=\poly\left(k\right)$, $q=\poly\left(\ell\right)$, and $u,v\in\mathbb{M}$.
If $u\sig\left(k\right)=v\sig\left(\ell\right)$, then $v\poly\left(\ell\right)$
is rewritable.\end{lem}
\begin{proof}
Assume that $u\sig\left(k\right)=v\sig\left(\ell\right)=\multiplier\basisvar_{\indexvar}$
for some $\multiplier\in\mathbb{M}$, $\indexvar\in\left\{ 1,\ldots,m\right\} $.
Since the signature indices are equal at $\indexvar$ and $k>\ell$,
$p$ is a rewriting of an $S$-poly\-nomial indexed by $\rules_{\indexvar}$,
so $\left(\sig\left(k\right),k\right)$ appears in $\rules_{\indexvar}$
after $\left(\sig\left(\ell\right),\ell\right)$ (assuming that $\left(\sig\left(\ell\right),\ell\right)$
appears at all, which it will not if $\ell=\indexvar$). Hence $\findrewriting\left(v,\ell\right)\neq\ell$,
$\rewritable\left(v,\ell\right)=\verb|true|$, and $v\sig\left(\ell\right)$
is rewritable.
\end{proof}

\subsection{\label{sub:Rewriting,-syzygies,-and}New Characterization of a Gr\"obner
Basis.}
\begin{defn}
\label{def:syzygy}A \emph{syzygy} of $F$ is some $\mathbf{H}\in\polyring^{m}$
such that $\mathbf{H}\cdot F=H_{1}f_{1}+\cdots+H_{m}f_{m}=0$.\end{defn}
\begin{prop}
\label{prop:non-min sig only if syz}Suppose that $\multiplier\basisvar_{\indexvar}$
is a signature of some $p\in\polyring$, and $\mathbf{h}$ a corresponding
$F$-re\-pre\-sen\-ta\-tion. If $\multiplier\basisvar_{\indexvar}$
is not the minimal signature of $p$, then there exists a syzygy $\mathbf{H}$
of $F$ satisfying each of the following:
\begin{lyxlist}{(B)}
\item [{(A)}] $\multiplier\basisvar_{\indexvar}$ is a signature of $\mathbf{H}\cdot F$,
and
\item [{(B)}] $\left(\mathbf{h}-\mathbf{H}\right)$ is an $F$-representation
of $p$ corresponding to the minimal signature.
\end{lyxlist}
\end{prop}
\begin{proof}
Assume that $\multiplier\basisvar_{\indexvar}$ is not the minimal
signature of $p$. Suppose that $\othermultiplier\basisvar_{\otherindexvar}$
is the minimal signature of $p$. Then $\otherindexvar\leq\indexvar$.
By definition of a signature, there exists $\mathbf{h}\in\polyring^{m}$
such that\[
p=h_{1}f_{1}+\cdots+h_{\indexvar}f_{\indexvar},\]
and $\lt\left(h_{\indexvar}\right)=\multiplier$. Likewise, there exists $\mathbf{h}'\in\polyring^{m}$
such that\[
p=h_{1}'f_{1}+\cdots+h_{\otherindexvar}'f_{\otherindexvar},\]
and $\lt\left(h'_{\otherindexvar}\right)=\othermultiplier$. Let\[
H_{\lambda}=\begin{cases}
h_{\lambda}-h_{\lambda}',\quad & 1\leq\lambda\leq\otherindexvar\\
h_{\lambda}, & \otherindexvar<\lambda\leq\indexvar\\
0, & \indexvar<\lambda\leq m\end{cases}\]
for each $\lambda=1,2,\ldots,m$; then\[
0=p-p=\sum_{\lambda=1}^{m}H_{\lambda}f_{\lambda}.\]
Let $\mathbf{H}=\left(H_{1},\ldots,H_{m}\right)$; observe that
\begin{itemize}
\item $\mathbf{H}$ is a syzygy of $F$;
\item $\othermultiplier\basisvar_{\otherindexvar}\siglt\multiplier\basisvar_{\indexvar}$
implies that

\begin{itemize}
\item $h_{\otherindexvar+1}-H_{\otherindexvar+1}=\cdots=h_{\indexvar}-H_{\indexvar}=0$
and $\lt\left(h_{\otherindexvar}-H_{\otherindexvar}\right)=\lt\left(h_{\otherindexvar}'\right)=\othermultiplier$;
\item $\lt\left(H_{\indexvar}\right)=\multiplier$, so $\multiplier\basisvar_{\indexvar}$
is a signature of $\mathbf{H}\cdot F$, satisfying (A); so
\item $\mathbf{h}-\mathbf{H}=\mathbf{h'}$, satisfying (B).
%\item $\left(\mathbf{h}-\mathbf{H}\right)\cdot F$ is an $F$-representation
%of $p$ corresponding to the minimal signature, satisfying (B).
\end{itemize}
\end{itemize}
\end{proof}
Inspection of the algorithms that assign signatures to polynomials
shows that F5 \emph{attempts} to assign the minimal signature with
respect to $F$ of each labeled polynomial in $\lp$:
\begin{itemize}
\item the signature assigned to each $f_{i}$ of the input is $\basisvar_{i}$;
\item the signatures assigned to $S$-poly\-nomials are, by Proposition
\ref{pro:props of sigs}, the smallest one can predict from the information
known; and
\item if top-reduction would increase a polynomial's signature, then $\topreduction$
generates a new $S$-poly\-nomial with that signature, preserving
the signature of the current polynomial.
\end{itemize}
This does not always succeed, but Theorem~\ref{thm:F5 characterization}
implies a benefit.
\begin{thm}
[New characterization]\label{thm:F5 characterization}Suppose that
iteration $i$ of  $\partialbasis$ terminates with output $\currbasis$.
Let $\mathcal{G}=\left(\poly\left(\lambda\right):\;\lambda\in\currbasis\right)$.
If every $S$-poly\-nomial $S$ of $\mathcal{G}$ satisfies (A) or
(B) where
\begin{lyxlist}{(B)}
\item [{(A)}] $S$ has a signature-preserving $S$-re\-pre\-sen\-ta\-tion
with respect to $\mathcal{G}$;
\item [{(B)}] a component $u\poly\left(k\right)$ of $S$ satisfies

\begin{lyxlist}{(A2)}
\item [{(B1)}] $u\sig\left(k\right)$ has signature index $i$ but is not
the minimal signature of $u\poly\left(k\right)$; or
\item [{(B2)}] $u\sig\left(k\right)$ is rewritable in $\rules$;
\end{lyxlist}
\end{lyxlist}
\noindent then $\mathcal{G}$ is a Gr\"obner basis of $\left\langle F_{i}\right\rangle $.\end{thm}
\begin{rem*}
Faug\`ere and Stegers prove a theorem similar to that of Theorem~\ref{thm:F5 characterization}
(Theorem 1 in \citep{Fau02Corrected}; Theorem~3.21 in \citep{Stegers2006}),
but their formulation of the theorem does not consider (B2), and their
notion of a component's not being ``normalized'' is less general and
not quite the same as (B1).\end{rem*}
\begin{proof}
Let $S$ be any $S$-poly\-nomial of $\mathcal{G}$, and $t$ the
head term of either component of $S$. The components of $S$ define
a $\mathcal{G}$-re\-pre\-sen\-ta\-tion $\mathbf{h}$ of $S$.
This initial $\mathbf{h}$ is not an $S$-re\-pre\-sen\-ta\-tion
of $S$; we will rewrite $\mathbf{h}$ repeatedly until it is. As
long as it is not, we know that there exist $j,k$ such that $\lt\left(h_{j}\mathcal{G}_{j}\right)=\lt\left(h_{k}\mathcal{G}_{k}\right)\ordge t$;
any such pair corresponds to what we call ``intermediate $S$-poly\-nomials'':
\begin{enumerate}
\item \label{enu: iteration 1}Let $S'$ be the intermediate $S$-poly\-nomial
whose natural signature is maximal among all natural signatures of
intermediate $S$-poly\-nomials. There may be a choice of $S$-poly\-nomials;
if so, choose $j,k\in\bases_{i}$ such that for any other $\ell\in\bases_{i}$
such that $\lt\left(h_{\ell}\mathcal{G}_{\ell}\right)=\lt\left(h_{j}\mathcal{G}_{j}\right)=\lt\left(h_{k}\mathcal{G}_{k}\right)$,
we have $\lt\left(h_{j}\right)\sig\left(j\right)>\lt\left(h_{k}\right)\sig\left(k\right)>\lt\left(h_{\ell}\right)\sig\left(\ell\right)$.
Then:

\begin{itemize}
\item If $S'$ satisfies (A), use a signature-preserving $S$-representation
to rewrite $\mathbf{h}$.
\item If a component of $S'$ satisfies (B1), use the syzygy identified
by Proposition~\ref{prop:non-min sig only if syz} to rewrite $\mathbf{h}$
with the minimal signature.
\item If a component of $S'$ satisfies (B2), use Lemma~\ref{pro: structure of rewriting}
with the rewriter of maximal index in $\rules$ to rewrite $u\poly\left(k\right)$,
and thus $\mathbf{h}$.
\end{itemize}
\item Is the rewritten $\mathbf{h}$ an $S$-re\-pre\-sen\-ta\-tion
of $S$? If so, stop. If not, there exist intermediate $S$-polynomials
in the $\mathcal{G}$-representation of $S$. Return to \eqref{enu: iteration 1}.
\end{enumerate}
\noindent We claim that the iterative process outlined above terminates
with an $S$-representation of $S$. Why? Let $\mathcal{M}$ be the
larger natural signature of a component of $S'$, and $\mathcal{N}$
the smaller natural signature of a component of $S'$.
\begin{itemize}
\item In case (A), the signature-preserving representation guarantees that
any component of a newly introduced intermediate $S$-poly\-nomial
has a natural signature smaller than $\mathcal{M}$, except possibly
one, $d\sig\left(\ell\right)$ for some $\ell\in\bases_{i}$ and some
$d\in\termset$. By Definition~\ref{def:sig-preserving S-rep}, $\ell>k$,
where $k$ is the largest index in $\bases_{i}$ of a generator of
$S'$. 
\item If in case (B1), Lemma~\ref{prop:non-min sig only if syz} implies
that the component is rewritten with a lower signature.
\item If in case (B2), suppose, without loss of generality, that $u\poly\left(k\right)$
is the component of $S'$ that is rewritable. Denote its rewriter
by $\poly\left(\ell\right)$ for some $\ell\in\bases_{i}$. Lemma~\ref{pro: structure of rewriting}
implies any polynomials introduced by the rewriting have smaller signature
than $u\sig\left(k\right)$ except $d\poly\left(\ell\right)$, where
$d\in\termset$ such that $d\sig\left(\ell\right)=u\sig\left(k\right)$.
By Definitions~\ref{def: list of rewritings} and~\ref{def: rewritable},
$\ell>k$. We chose the rewriter of maximal index in $\rules$, so
$d\poly\left(\ell\right)$ is not itself rewritable.
\end{itemize}
In most cases, $\mathcal{M}$ will not increase. There is one exception:%
\footnote{Thanks to Vasily Galkin for pointing out this exception.%
} if a (B1) or (B2) rewriting is applied to the component with natural
signature $\mathcal{N}$, it may happen that the head term of the
component corresponding to $\mathcal{M}$ cancels a \emph{non}-head
term of a polynomial $\mathcal{G}_{\lambda}$ introduced by the rewriting.
In this case, let $S''$ be the intermediate $S$-poly\-nomial whose
natural signature $\mathcal{M}'$ is maximal among all natural signatures
of intermediate $S$-poly\-nomials. This gives rise to a possible
recursion that nevertheless terminates; after all, $\mathcal{M}'\siglt\mathcal{M}$
and $\siglt$ is a well-ordering. Moreover, once we return to an intermediate
$S$-poly\-nomial of natural signature $\mathcal{M}$, then if we
were in case (B1), the corresponding value of $\mathcal{N}$ is smaller,
whereas in case (B2), the index $k$ of the component with natural
signature $\mathcal{N}$ is larger. By the well-ordering property
of $\siglt$, $\mathcal{N}$ can increase only finitely many times.
By the assumption that $\partialbasis$ terminated, $\bases_{i}$
is finite, so $k$ can increase only finitely many times. Hence $\mathcal{M}$
can increase to any previous value only finitely many times.

So each iteration either decreases one of $\mathcal{M}$ or $\mathcal{N}$,
or increases the index in $\lp$ of the polynomial with natural signature
$\mathcal{M}$ or $\mathcal{N}$. The choice of maximal index in a
(B2) rewriting implies that (B2) can be applied at most once for each
value of $\mathcal{M}$ or $\mathcal{N}$. Since $\mathcal{G}$ is
finite, the index of the component of natural signature $\mathcal{M}$
or $\mathcal{N}$ cannot increase indefinitely. Both (A) and (B2)
rewritings increase that index, so eventually any intermediate $S$-poly\-nomial
with natural signature $\mathcal{M}$ must have a signature-preserving
representation. In other words, $\mathcal{M}$ must eventually decrease
permanently below any given level.

By the well-ordering property of $\siglt$, $\mathcal{M}$ cannot
decrease indefinitely. Hence the iteration must terminate with an
$S$-repre\-sen\-tation of $S'$. Since $S$ was an arbitrary $S$-poly\-nomial
of $\mathcal{G}$, it must be that $\mathcal{G}$ is a Gr\"obner
basis of $\ideal{F_{i}}$.
\end{proof}

\subsection{\label{sub:Principal-syzygies}Principal Syzygies}

\label{rem:min sigs}Suppose that all syzygies of $F$ are generated
by principal syzygies of the form $f_{i}\basisvar_{j}-f_{j}\basisvar_{i}$.
If $\sig\left(k\right)$ is not minimal, then by Proposition~\ref{prop:non-min sig only if syz}
some monomial multiple of a principal syzygy $\mu\left(f_{i}\basisvar_{j}-f_{j}\basisvar_{i}\right)$
has the same signature as $\sig\left(k\right)$. This provides an
easy test for such a non-minimal signature.
\begin{defn}
We say that a polynomial multiple $u\poly\left(k\right)$ \emph{satisfies
Faug\`ere's criterion with respect to }$\prevbasis$ if
\begin{itemize}
\item $\sig\left(k\right)=\multiplier\basisvar_{\indexvar}$; and
\item there exists $\ell\in\prevbasis$ such that

\begin{itemize}
\item $\sig\left(\ell\right)=\othermultiplier\basisvar_{\otherindexvar}$
where $\otherindexvar<\indexvar$; and
\item $\lt\left(\poly\left(\ell\right)\right)$ divides $u\multiplier$.
\end{itemize}
\end{itemize}
\end{defn}
\begin{prop}
\label{pro: reduction gives smaller signatures}If a polynomial multiple
$u\poly\left(k\right)$ satisfies Faugere's criterion with respect
to $\prevbasis$ then $u\sig\left(k\right)$ is not the minimal signature
of $u\poly\left(k\right)$.\end{prop}
\begin{proof}
Assume that a polynomial multiple $u\poly\left(k\right)$ satisfies
Faugere's criterion with respect to $\prevbasis$. Let $p=\poly\left(k\right)$
and $\multiplier\basisvar_{\indexvar}=\sig\left(k\right)$, so there
exists $\mathbf{h}\in\polyring^{m}$ such that\[
p=h_{1}f_{1}+\cdots+h_{m}f_{m},\]
$h_{\indexvar+1}=\cdots=h_{m}=0$, and $\lt\left(h_{\indexvar}\right)=\multiplier$.
From the definition of Faug\`ere's criterion, there exists $\ell\in\prevbasis$
such that $\lt\left(\poly\left(\ell\right)\right)$ divides $u\multiplier$.
Let $q=\poly\left(\ell\right)$. Since $\ell\in\prevbasis$, there
exists $\mathbf{H}\in\polyring^{m}$ such that $\otherindexvar<\indexvar$,\[
q=H_{1}f_{1}+\cdots+H_{m}f_{m},\]
$H_{\otherindexvar+1}=\cdots=H_{m}=0$, and $H_{\indexvar'}\neq0$.
Choose $d\in\mathbb{M}$ such that $d\cdot\lt\left(q\right)=u\multiplier$.
Observe that\begin{align}
up & =u\left(h_{1}f_{1}+\cdots+h_{\indexvar}f_{\indexvar}\right)\nonumber \\
 & =u\left[\sum_{\lambda=1}^{\indexvar-1}h_{\lambda}f_{\lambda}+\left(h_{\indexvar}-\lt\left(h_{\indexvar}\right)\right)\cdot f_{\indexvar}\right]+u\lt\left(h_{\indexvar}\right)f_{\indexvar}.\label{eq: trivial syz 1}\end{align}
Let\[
P=u\left[\sum_{\lambda=1}^{\indexvar-1}h_{\lambda}f_{\lambda}+\left(h_{\indexvar}-\lt\left(h_{\indexvar}\right)\right)\cdot f_{\indexvar}\right];\]
equation \eqref{eq: trivial syz 1} becomes\begin{align}
up & =P+u\cdot\lt\left(h_{\indexvar}\right)f_{\indexvar}\nonumber \\
 & =P+\left(u\multiplier\right)\cdot f_{\indexvar}\nonumber \\
 & =P+\left(d\cdot\lt\left(q\right)\right)\cdot f_{\indexvar}\nonumber \\
 & =P+d\cdot\left(\lt\left(\sum_{\lambda=1}^{\otherindexvar}H_{\lambda}f_{\lambda}\right)\right)\cdot f_{\indexvar}.\label{eq: trivial syz 2}\end{align}
By the distributive and associative properties\[
\left(\sum_{\lambda=1}^{\otherindexvar}H_{\lambda}f_{\lambda}\right)f_{\indexvar}=\sum_{\lambda=1}^{\otherindexvar}f_{\lambda}\left(H_{\lambda}f_{\indexvar}\right),\]
so\[
\lt\left(\sum_{\lambda=1}^{\otherindexvar}H_{\lambda}f_{\lambda}\right)f_{\indexvar}=\sum_{\lambda=1}^{\otherindexvar}f_{\lambda}\left(H_{\lambda}f_{\indexvar}\right)-\left[\left(\sum_{\lambda=1}^{\otherindexvar}H_{\lambda}f_{\lambda}\right)-\lt\left(\sum_{\lambda=1}^{\otherindexvar}H_{\lambda}f_{\lambda}\right)\right]f_{\indexvar}.\]
Let\[
Q=\sum_{\lambda=1}^{\otherindexvar}f_{\lambda}\left(H_{\lambda}f_{\indexvar}\right)\quad\mbox{and}\quad R=\left[\left(\sum_{\lambda=1}^{\otherindexvar}H_{\lambda}f_{\lambda}\right)-\lt\left(\sum_{\lambda=1}^{\otherindexvar}H_{\lambda}f_{\lambda}\right)\right].\]
We can rewrite equation \eqref{eq: trivial syz 2} as\begin{align*}
up & =P+dQ-\left(dR\right)\cdot f_{\indexvar}.\end{align*}
We claim that we have rewritten $up$ with a signature smaller than
$\left(u\multiplier\right)\basisvar_{\indexvar}$. By construction,
$P$ has a signature smaller than $\left(u\multiplier\right)\basisvar_{\indexvar}$.
By inspection, $Q$ has a signature index no greater than $\indexvar'$,
so $dQ$ has a signature smaller than $\left(u\multiplier\right)\basisvar_{\indexvar}$.
That leaves $\left(dR\right)\cdot f_{\indexvar}$, and\begin{align*}
\lt\left(dR\right) & =d\cdot\lt\left[\left(\sum_{\lambda=1}^{\indexvar'}H_{\lambda}f_{\lambda}\right)-\lt\left(\sum_{\lambda=1}^{\indexvar'}H_{\lambda}f_{\lambda}\right)\right]\\
 & =d\cdot\lt\left(q-\lt\left(q\right)\right)\\
 & \ordlt d\cdot\lt\left(q\right)\\
 & \qquad=u\multiplier.\end{align*}
Hence $\left(dR\right)\cdot f_{\indexvar}$ has a signature smaller
than $\left(u\multiplier\right)\basisvar_{\indexvar}$, and $up$
has a signature smaller than $\left(u\multiplier\right)\basisvar_{\indexvar}$.
That is, $u\sig\left(k\right)$ is not the minimal signature of $up$.
\end{proof}
If a polynomial multiple $u\poly\left(k\right)$ satisfies Faugere's
criterion with respect to $\prevbasis$, then by Proposition~\ref{pro: reduction gives smaller signatures}
and Theorem~\ref{thm:F5 characterization} we need not compute it.
$\critpair$ and $\findreductor$ discard any polynomial multiple
that satisfies Faug\`ere's criterion. Thus Theorem~\ref{thm:F5 characterization}
and Proposition~\ref{pro: reduction gives smaller signatures} show
that:
\begin{cor}
\label{cor: F5 no red to zero on regular inputs}Given $F$, the output
of the F5 algorithm is a Gr\"obner basis of $\ideal{F}$. Also, if
all the syzygies of $F$ are principal, then F5 does not reduce any
polynomials to zero.
\end{cor}
Corollary \ref{cor: F5 no red to zero on regular inputs} does \emph{not}
imply:
\begin{itemize}
\item that F5 does not generate redundant polynomials. The example from
\citep{Fau02Corrected} generates one such polynomial ($r_{10}$).
\item that F5 terminates, at least not obviously. To the contrary, $\findreductor$
rejects potential reducers that are rewritable or that satisfy Faug\`ere's
criterion. As a result, the algorithm can compute a Gr\"obner basis,
while new polynomials that are not completely top-reduced continue
to generate new critical pairs. We have not observed an infinite loop in practice.
\end{itemize}

\subsection{\label{sec: Correctness}Correctness of the output of F5C}

We come now to the correctness of F5C. For correctness, we argue that
each stage of F5C imitates the behavior of F5 on an input equivalent
to the data structures generated by $\createreducedbasis$. Recall
that $F_{i}=\left(f_{1},\ldots,f_{i}\right)$. We will refer to the
system $F'=\left(\reducedbasis_{1},\ldots,\reducedbasis_{\#\reducedbasis},f_{i+1}\right)$
where
\begin{itemize}
\item $\reducedbasis$ is computed during the execution of $\createreducedbasis$;
and
\item $F'$ is indexed as $F'_{i}=\reducedbasis_{i}$, etc.
\end{itemize}
It is trivial that $\ideal{B}=\ideal{F_{i}}$ and $\ideal{F'}=\ideal{F_{i+1}}$.
\begin{lem}
\label{lem: admissible wrt F'}When  $\createreducedbasis$ terminates,
every element of $\lp$ is admissible with respect to $\reducedbasis$,
and thus with respect to $F'$.
\end{lem}
The proof is evident from inspection of $\createreducedbasis$.

The correctness of the behavior of $\rewritable$ in F5C hinges on
Definition~\ref{def: rewritable by zero}.
\begin{defn}
\label{def: rewritable by zero}Let $t\in\mathbb{M}$ and $k\in\currbasis$.
At any point in the algorithm, we say that a polynomial multiple $t\poly\left(k\right)$\emph{
}is \emph{rewritable by the zero polynomial} if there exist $a,b\in\prevbasis$
such that
\begin{itemize}
\item the $S$-poly\-nomial $S$ of $p=\poly\left(a\right)$ and $q=\poly\left(b\right)$
reduces to zero, although the reduction may not be signature-safe;
and
\item $\max\left(\ltmul{p,q}\sig\left(a\right),\ltmul{q,p}\sig\left(b\right)\right)$
divides $t\sig\left(k\right)$.
\end{itemize}
\end{defn}
\begin{rem*}
It is essential that $a,b\in\prevbasis$ and not in $\currbasis$.
The fact that a component of an $S$-poly\-nomial is rewritable does
\emph{not} imply that it is rewritable by the zero polynomial. Proposition~\ref{pro: structure of rewriting}
implies that if a component of an $S$-poly\-nomial is rewritable,
then the $S$-poly\-nomial can be rewritten using a polynomial of
the same signature; \emph{however,} the resulting $S$-repre\-sen\-tation
may not yet exist when the component is detected to be rewritable.\end{rem*}
\begin{lem}
\label{lem: behavior of is_rewritable equivalent}When  $\createreducedbasis$
terminates,  $\rewritable$ in F5C would return \verb|true| for the
input $\left(u,k\right)$, only if $u\poly\left(k\right)$ is rewritable
by the zero polynomial.\end{lem}
\begin{proof}
\noindent Line~\ref{line: compute reduced basis} of $\createreducedbasis$
interreduces the polynomials indexed by $\prevbasis$ to obtain the
reduced Gr\"obner basis $\reducedbasis$. Thus all $S$-poly\-nomials
of $\reducedbasis$ reduce to zero. When $\createreducedbasis$
terminates, $\rules$ consists of a list of lists. Elements of the
$j$th list have the form $\omega_{k}=\left(\ltmul{\reducedbasis_{j},\reducedbasis_{k}},0\right)$
for $k=1,\ldots,j-1$ where, as explained in the introduction,\[
\ltmul{\reducedbasis_{j},\reducedbasis_{k}}=\frac{\lcm\left(\lt\left(\reducedbasis_{j}\right),\lt\left(\reducedbasis_{k}\right)\right)}{\lt\left(\reducedbasis_{j}\right)}.\]
Thus if $\rewritable\left(u,k\right)$ is true, then $u\poly\left(k\right)$
is rewritable by the zero polynomial.\end{proof}
\begin{cor}
\label{cor: isrewritable still valid}In F5C, if $\rewritable$ returns
\verb|true| for the input $\left(u,k\right)$, then $u\poly\left(k\right)$
is rewritable either by a polynomial that appears in $\lp$, or by
the zero polynomial.\end{cor}
\begin{proof}
This is evident from consequence of Proposition~\ref{pro: algorithm rewritable detects rewritables}
and the isolation of all modifications of F5 to $\createreducedbasis$.
Assume that $\rewritable$ returns \verb|true| for $\left(u,k\right)$.
Let $j=\findrewriting\left(u,k\right)$. %If $j=0$, Lemma~\ref{lem: behavior of is_rewritable equivalent}
%implies that $u\poly\left(k\right)$ is rewritable by the zero polynomial.
If $j=0$, then $\createreducedbasis$ added $(u,0)$ to $\rules$.
No other algorithm adds a pair of the form $(u,0)$ to $\rules$,
so by Lemma~\ref{lem: behavior of is_rewritable equivalent}, $u\poly(k)$ is
rewritable by the zero polynomial.
Otherwise, line~\ref{line: start to create new rules} of $\createreducedbasis$
implies that $j\in\currbasis\backslash\prevbasis$. That is, $\lp_{j}$
was generated in the same way that F5 would generate it. By Proposition~\ref{pro: algorithm rewritable detects rewritables},
$u\poly\left(k\right)$ is rewritable by $\poly\left(j\right)$.\end{proof}
\begin{thm}
\label{lem: termination implies basis}If $\partialbasisc$ terminates
for a given input $i$, then it terminates with a Gr\"obner basis
of $\ideal{F_{i}}$.\end{thm}
\begin{proof}
The proof is adapted easily from the proof of Theorem~\ref{thm:F5 characterization},
using Lemma~\ref{lem: admissible wrt F'} and Corollary~\ref{cor: isrewritable still valid}.
In particular, $S$-poly\-nomials that are rewritable by the zero
polynomial---that is, the $S$-poly\-nomials of $\reducedbasis$---can
be rewritten in the same manner as polynomials that satisfy case (A)
of Theorem~\ref{thm:F5 characterization}.
\end{proof}
Changing the algorithm's point of view so that some polynomials are
admissible with respect to $F'$ and not to $F$ implies the possibility
of introducing non-principal syzygies. Of course we would like F5C
to avoid any reductions to zero that F5 also avoids; otherwise the
benefit from a reduced Gr\"obner basis could be offset by the increased
cost of wasted computations. Hence we must show that if the syzygies
of the input $F$ are all principal, then F5C does not introduce reductions
to zero. Lemma~\ref{lem: admissible wrt F after translation} shows
that the signature of a polynomial indexed by $\currbasis\backslash\prevbasis$
in F5C corresponds to the signature that F5 would compute, ``translated''
by $\#\reducedbasis-\left(i-1\right)$.
\begin{lem}
\label{lem: admissible wrt F after translation}Let $i>2$. During
the $i$th pass through the while loop of $\basis$, let $k\in\currbasis\backslash\prevbasis$,
and $\sig\left(k\right)=\multiplier\basisvar_{\indexvar}$, where
the signature is with respect to $F'$. Then $\left(\multiplier\basisvar_{i},\poly\left(k\right)\right)$
is admissible with respect to $F$.\end{lem}
\begin{proof}
From the assumption that $\sig\left(k\right)=\multiplier\basisvar_{\indexvar}$,
we know that there exists $\mathbf{h}\in\polyring$ such that\[
\poly\left(k\right)=\sum_{\lambda=1}^{m}h_{\lambda}F'_{\lambda},\]
$h_{\indexvar+1}=\cdots=h_{m}=0$, and $\lt\left(h_{\indexvar}\right)=\multiplier$.
Recall that $F'_{\lambda}=\reducedbasis_{\lambda}$ for each $\lambda=1,\ldots,\indexvar-1$.
By Theorem~\ref{lem: termination implies basis}, there exist $H_{1},\ldots,H_{i-1}$
such that\[
\sum_{\lambda=1}^{\indexvar-1}h_{\lambda}F'_{\lambda}=\sum_{\lambda=1}^{m}H_{\lambda}f_{\lambda}\]
and $H_{\indexvar}=\cdots=H_{m}=0$. In addition, $F_{\indexvar}'=f_{i}$.
Hence\[
\poly\left(k\right)=\sum_{\lambda=1}^{\indexvar-1}H_{\lambda}f_{\lambda}+h_{\indexvar}f_{i},\]
whence $\left(\multiplier\basisvar_{i},\poly\left(k\right)\right)$
is admissible with respect to $F$.\end{proof}
\begin{thm}
\label{cor: no reduction to zero if F regular}If the syzygies of
$F$ are all principal syzygies, then F5C does not reduce any polynomial
to zero.\end{thm}
\begin{proof}
Assume for the contrapositive that $k\in\currbasis$ and the algorithm
reduces $\poly\left(k\right)$ to zero. Suppose that we are on iteration
$i$ of the \verb|while| loop of $\basis$. Let $\sig\left(k\right)=\multiplier\basisvar_{\indexvar}$.
This signature of $\poly\left(k\right)$ is with respect to $F'$;
from Lemma~\ref{lem: admissible wrt F after translation} we infer
that $\multiplier\basisvar_{i}$ is a signature of $\poly\left(k\right)$
with respect to $F$.

Now $\prevbasis$ indexes a reduced Gr\"obner basis $\reducedbasis$
of $\ideal{F_{i}}$. The reduction to zero implies that $\critpair$
did not discard the corresponding critical pair, which in turn implies
that no head monomial of $\reducedbasis$ divided $\multiplier$.
By the definition of a reduced Gr\"obner basis, no head monomial
of the unreduced basis would have divided $\multiplier$ either. By
Corollary~\ref{cor: F5 no red to zero on regular inputs}, the syzygies
of $F$ are not principal.\end{proof}
\begin{rem*}
In our experiments with inputs whose syzygies are not principal, it
remains the case that F5C computes no more reductions to zero than
does F5. However, we do not have a proof of this. The difficulty lies
in the fact that signatures of polynomials in $\prevbasis$ need not
be the same in F5 and F5C. F5 computes different critical pairs, which
may generate different rewrite rules. This introduces the possibility
that F5 rejects some polynomials as rewritable that F5C does not.
However, we have not observed this in practice.
\end{rem*}
We conclude with two final, surprising results.
\begin{thm}
\label{thm: can skip rewrite rules for B}In $\createreducedbasis$,
there is no need to recompute the rewrite rules for $\reducedbasis$.\end{thm}
\begin{proof}
When performing top-reductions by elements of $\reducedbasis$, the
algorithm checks neither whether a polynomial multiple is rewritable,
nor whether it satisfies Faug\`ere's criterion. Thus we only need
to verify the statement of the theorem in the context of $S$-poly\-nomial
creation. Suppose therefore that we are computing $\bases_{i}$, the
Gr\"obner basis of $\left\langle F_{i}'\right\rangle $ where $F_{i}'=\left(\reducedbasis_{1},\ldots,\reducedbasis_{\#\reducedbasis},f_{i}\right)$,
and while computing the $S$-poly\-nomial of $p=\poly\left(k\right)$
and $q=\poly\left(\ell\right)$, where $k\in\currbasis\backslash\prevbasis$
and $\ell\in\prevbasis$, $\rewritable$ reports that $\ltmul{q,p}q$
is rewritable.

We claim that it will also reject $\ltmul{p,q}p$. Since $\ltmul{q,p}q$
is rewritable, there exists $j\in\left\{ 1,\ldots,\#\reducedbasis\right\} $
such that\[
\frac{\lcm\left(\lt\left(q\right),\left(\lt\left(\reducedbasis_{j}\right)\right)\right)}{\lt\left(q\right)}\quad\mbox{divides}\quad\frac{\lcm\left(\lt\left(p\right),\lt\left(q\right)\right)}{\lt\left(q\right)}.\]
It follows that $\lcm\left(\lt\left(q\right),\lt\left(\reducedbasis_{j}\right)\right)$
divides $\lcm\left(\lt\left(p\right),\lt\left(q\right)\right)$. %A
%straightforward argument on the degrees of the variables implies that
%$\lcm\left(\lt\left(p\right),\lt\left(\reducedbasis_{j}\right)\right)$
%also divides $\lcm\left(\lt\left(p\right),\lt\left(q\right)\right)$.
Thus\[
\frac{\lcm\left(\lt\left(p\right),\lt\left(\reducedbasis_{j}\right)\right)}{\lt\left(p\right)}\mbox{ divides }\frac{\lcm\left(\lt\left(p\right),\lt\left(q\right)\right)}{\lt\left(p\right)}=\ltmul{p,q}.\]
The design of the algorithm implies that the $S$-poly\-nomial of
$p$ and $\reducedbasis_{j}$ would have been considered before the
$S$-poly\-nomial of $p$ and $q$. This leads to two possibilities.
\begin{enumerate}
\item The $S$-poly\-nomial of $p$ and $\reducedbasis_{j}$ was computed,
so that the rewrite rule $\left(\ltmul{p,\reducedbasis_{j}},\lambda\right)$
appears in $\rules_{i}$ for some $\lambda\in\currbasis$. Hence $\rewritable\left(\ltmul{p,q},k\right)$
returns \verb|true|.
\item The $S$-poly\-nomial of $p$ and $\reducedbasis_{j}$ was rejected,
either because $\ltmul{p,\reducedbasis_{j}}p$ is rewritable or because
it satisfies Faug\`ere's criterion. Either one implies that the $\ltmul{p,q}p$
will also be rejected.
\end{enumerate}
\noindent Hence there is no need to compute the rewrite rules for
$\reducedbasis$.\end{proof}
\begin{cor}
\label{cor: F5C w/out sigs of B}We can reformulate F5C so that $\createreducedbasis$
is unnecessary, and the list $\rules$ records only signatures of
polynomials indexed by $\currbasis\backslash\prevbasis$.\end{cor}
\begin{proof}
Theorem~\ref{thm: can skip rewrite rules for B} implies that we
do not need the signatures of polynomials indexed by $\prevbasis$
for the rewrite rules. In fact, this is the only reason we might need
their signatures, since $\spol$ always uses the larger signature
to create an $S$-poly\-nomial, and $\topreduction$ top-reduces
by $\reducedbasis$ without checking signatures. Hence the signatures
of polynomials indexed by $\prevbasis$ are useless.

We now indicate
how to revise the algorithm to take this into account.
As in F5R, replace line~\ref{line: definition of B} of $\basisoriginal$
with

\begin{center}
\ref{line: definition of B}\algindent Let $\reducedbasis$ be the
interreduction of $\left\{ \poly\left(\lambda\right):\;\lambda\in\prevbasis\right\} $.
\par\end{center}

\noindent Subsequently, change line~\ref{line: critpair, PS-reducible 1}
of $\critpair$ to

\begin{center}
\ref{line: critpair, PS-reducible 1}\algindent \textbf{if }$k\not\in\prevbasis$
\textbf{and} $u_{1}\cdot\multiplier_{1}$ is top-reducible by $\prevbasis$
\par\end{center}

\noindent and line~\ref{line: critpair, PS-reducible 2} of $\critpair$
to

\begin{center}
\ref{line: critpair, PS-reducible 2}\algindent \textbf{if} $\ell\not\in\prevbasis$
\textbf{and} $u_{2}\cdot\multiplier_{2}$ is top-reducible by $\prevbasis$
\par\end{center}

\noindent Similarly adjust line~\ref{line: spol rewritiable?} of
$\spol$ and line~\ref{line: FindReductor test} of $\findreductor$
so that they do not check polynomials of $\prevbasis$. Modify the
definition of $\rules$ so that it is only one list, not a list of
lists, and $\findrewriting$ so that it only searches backwards through
$\rules$, rather than finding which list in $\rules$ to check. Theorem~\ref{thm: can skip rewrite rules for B}
implies that if the original F5C terminates correctly, then this modified
version of F5C also terminates correctly.\end{proof}
\begin{rem*}
Theorem~\ref{thm: can skip rewrite rules for B} applies only to
F5C, not to F5. The difference is that for any $\ell\in\prevbasis$,
F5C guarantees that $\sig\left(\ell\right)=\sigformat{\multiplier}{\ell}$
where $\multiplier=1$. This is not the case in F5.

The prototype implementations of F5C are primarily for educational
purposes, so for the sake of clarity we implement the given pseudocode
without the optimization outlined in the proof of Corollary~\ref{cor: F5C w/out sigs of B}.
\end{rem*}

\subsubsection*{Acknowledgments}

The authors wish to thank the Centre for Computer Algebra at Universit\"at
Kaiserslautern for their hospitality, encouragement, and assistance
with the \textsc{Singular} computer algebra system. They would also
like to thank Martin Albrecht, who made a number of helpful comments
regarding the paper.

\section*{Appendix: Using the \textsc{Singular} and Sage prototype implementations}

The \textsc{Singular} prototype implementation contains three functions
\texttt{basis}, \texttt{basis\_r}, and \texttt{basis\_c} to compute
the Gr\"obner basis of an ideal. An example run with \texttt{basis}
is shown in Figure~\ref{fig: Singular example}. %
\begin{figure}
\caption{\label{fig: Singular example}Example run of the Singular prototype
system}

\texttt{> LIB {}``f5\_library.lib'';}

\texttt{// {*}{*} loaded f5\_library.lib (1.1,2009/01/26\textquotedbl{})}

\texttt{> ring R = 0,(x,y,z,t),dp;}

\texttt{> ideal i = yz3 - x2t2, xz2 - y2t, x2y - z2t;}

\texttt{> ideal B = basis(i);}

\texttt{Iteration 2}

\texttt{Processing 1 critical pairs of degree 5}

\texttt{Processing 1 critical pairs of degree 7}

\texttt{4 polynomials in basis}

\texttt{Iteration 3}

\texttt{Processing 1 critical pairs of degree 5}

\texttt{Processing 1 critical pairs of degree 6}

\texttt{Processing 4 critical pairs of degree 7}

\texttt{Processing 1 critical pairs of degree 8}

\texttt{10 polynomials in basis}

~

\texttt{number of zero reductions: 0}

\texttt{number of elements in g: 10}

\texttt{cpu time for gb computation: 50/1000 sec}

\texttt{> B;}

\texttt{B{[}1{]}=yz3-x2t2}

\texttt{B{[}2{]}=x2y-z2t}

\texttt{B{[}3{]}=xz2-y2t}

\texttt{B{[}4{]}=xy3t-z4t}

\texttt{B{[}5{]}=z6t-y5t2}

\texttt{B{[}6{]}=y3zt-x3t2}

\texttt{B{[}7{]}=z5t-x4t2}

\texttt{B{[}8{]}=y5t2-x4zt2}

\texttt{B{[}9{]}=x5t2-y2z3t2}

\texttt{B{[}10{]}=y6t2-xy2zt4}

\texttt{> }
\end{figure}
 While computing the Gr\"obner basis, this implementation also prints
for each degree the size of $\critpairs_{d}$, the set of critical
pairs passed to $\spol$. This implementation checks in both $\critpair$
and $\spol$ for the rewritten criterion, so $\#\critpairs_{d}$ is
sometimes smaller here than in Faug\`ere's paper, but the reader
can compare the results to see that the same basis is generated. A
large number of benchmark systems can be obtained by downloading the
companion file

\begin{center}
\texttt{http://www.math.usm.edu/perry/Research/f5ex.lib} .
\par\end{center}

\noindent For a further introduction to \textsc{Singular}, see \citep{GreuelPfister2008}.

The Sage prototype implementation contains four classes, \texttt{F5},
\texttt{F5R}, \texttt{F5C}, and \texttt{F4F5}. These can be called
by creating the appropriate class with a Sage ideal. An example run
with \texttt{F4F5} is shown in Figure~\ref{fig: Sage example}.%
\begin{figure}
\caption{\label{fig: Sage example}Example run of the Sage prototype implementation}

\texttt{sage: attach \textquotedbl{}/home/perry/common/Research/SAGE\_programs/f5.py\textquotedbl{}}

\texttt{sage: f5 = F4F5()}

\texttt{sage: R.<x,y,z,t> = QQ{[}{]}}

\texttt{sage: I = R.ideal(y{*}z\textasciicircum{}3-x\textasciicircum{}2{*}t\textasciicircum{}2,
x{*}z\textasciicircum{}2 - y\textasciicircum{}2{*}t, x\textasciicircum{}2{*}y
- z\textasciicircum{}2{*}t)}

\texttt{sage: B = f5(I)}

\texttt{Increment 1}

\texttt{1 critical pairs}

\texttt{Processing 1 pairs of degree 5 of 1 total}

\texttt{1 polynomials generated}

\texttt{1 x 2, 1, 0}

\texttt{1 polynomials left}

\texttt{Processing 1 pairs of degree 7 of 1 total}

\texttt{1 polynomials generated}

\texttt{1 x 2, 1, 0}

\texttt{1 polynomials left}

\texttt{Ended with 4 polynomials}

\texttt{Increment 2}

\texttt{4 critical pairs}

\texttt{Processing 1 pairs of degree 5 of 4 total}

\texttt{1 polynomials generated}

\texttt{1 x 2, 1, 0}

\texttt{1 polynomials left}

\texttt{Processing 2 pairs of degree 6 of 6 total}

\texttt{1 polynomials generated}

\texttt{1 x 2, 1, 0}

\texttt{1 polynomials left}

\texttt{Processing 4 pairs of degree 7 of 6 total}

\texttt{2 polynomials generated}

\texttt{4 x 6, 4, 0}

\texttt{2 polynomials left}

\texttt{Processing 2 pairs of degree 8 of 2 total}

\texttt{1 polynomials generated}

\texttt{2 x 3, 2, 0}

\texttt{1 polynomials left}

\texttt{Ended with 10 polynomials}

\texttt{sage: B}

\texttt{{[}x{*}z\textasciicircum{}2 - y\textasciicircum{}2{*}t,}

\texttt{x\textasciicircum{}2{*}y - z\textasciicircum{}2{*}t,}

\texttt{x{*}y\textasciicircum{}3{*}t - z\textasciicircum{}4{*}t,}

\texttt{y{*}z\textasciicircum{}3 - x\textasciicircum{}2{*}t\textasciicircum{}2,}

\texttt{y\textasciicircum{}3{*}z{*}t - x\textasciicircum{}3{*}t\textasciicircum{}2,}

\texttt{z\textasciicircum{}5{*}t - x\textasciicircum{}4{*}t\textasciicircum{}2,}

\texttt{y\textasciicircum{}5{*}t\textasciicircum{}2 - x\textasciicircum{}4{*}z{*}t\textasciicircum{}2,}

\texttt{x\textasciicircum{}5{*}t\textasciicircum{}2 - z\textasciicircum{}2{*}t\textasciicircum{}5{]}}

\texttt{sage: }
\end{figure}
 As in the \textsc{Singular} implementation, run-time data is printed.
In this case, the number of critical pairs in $\critpairs_{d}$, the
number of polynomials generated by $\spol$, and the size of the matrix
used for Gaussian elimination. No special techniques are used for
sparse matrices in this version, so it is rather slow (in fact, it
is slower than the other \texttt{F5}'s). The reader should notice
that in this version, the output has been interreduced, so there are
only~8 polynomials in the final result. For more information on Sage,
visit

\begin{center}
\texttt{http://www.sagemath.org/ }.
\par\end{center}

\bibliographystyle{elsart-harv}
\bibliography{/home/perry/common/Research/researchbibliography}

\end{document}